\documentclass{amsart}

\usepackage{amsmath,amsthm,amssymb,mathrsfs,color,enumerate,accents,amsfonts,bbold,adjustbox,etoolbox,changepage,bbm,adjustbox,witharrows,microtype}
\usepackage{theoremref}
\usepackage[margin = 1in]{geometry}

\usepackage{pgfplots}
\pgfplotsset{compat=1.15}
\usepgfplotslibrary{fillbetween} 
\usepackage{tikz-cd}
\usepackage{tikz}
\usetikzlibrary{tqft,knots,decorations.markings,arrows,external,hobby,positioning}

\usepackage{graphicx,xcolor,transparent}
\graphicspath{./Pictures/}

\usepackage{hyperref}
\hypersetup{
    colorlinks=true, 
    linktoc=all, 
    linkcolor=blue, 
}

\newtheorem{thm}{Theorem}[subsection]
\newtheorem*{thm*}{Theorem}
\newtheorem{cor}{Corollary}[subsection]
\newtheorem{lem}{Lemma}[subsection]

\theoremstyle{definition}
\newtheorem{dff}{Definition}[subsection]
\newtheorem{xmp}{Example}[subsection]
\newtheorem{rmk}{Remark}

\newcommand{\mbb}[1]{\mathbb{#1}}
\newcommand{\mcal}[1]{\mathcal{#1}}

\definecolor{darkgreen}{RGB}{26,190,26}

\title{Involutory Hopf group-coalgebras and invariants of flat bundles over 4-manifolds}
\author{Nicolas Bridges}
\address{Department of Mathematics \\ 150 N University St \\ Purdue University \\ West Lafayette, IN \\ 47907}
\email{bridge18@purdue.edu}

\author{Shawn X. Cui}
\address{Departments of Mathematics and Department of Physics and Astronomy \\ 150 N University St \\ Purdue University \\ West Lafayette, IN \\ 47907}
\email{cui177@purdue.edu}

\date{\today}

\begin{document}

\begin{abstract}
    We give invariants of flat bundles over 4-manifolds generalizing a result by Chaidez, Cotler, and Cui (Alg. \& Geo. Topology '22). We utilize a structure called a Hopf $G$-triplet for $G$ a group, which generalizes the notion of a Hopf triplet by Chaidez, Cotler, and Cui. In our construction, we present flat bundles over 4-manifolds using colored trisection diagrams: a direct analogue of colored Heegaard diagrams as described by Virelizier. Our main result is that involutory Hopf $G$-triplets of finite type yield well-defined invariants of $G$-colored trisection diagrams, and that if the monodromy of a flat bundle has image in $G$ we obtain invariants of flat bundles. We also show that a special Hopf $G$-triplet yields the invariant from Hopf $G$-algebras described by Mochida, thus generalizing the construction.
\end{abstract}

\maketitle

\section{Introduction}
\label{Introduction_Sec}
Hopf algebras are a rich source for constructing quantum invariants of three and four manifolds. Two fundamental families of quantum invariants of 3-manifolds are the Kuperberg invariant \cite{Kup91, kuperberg1997non} based on finite dimensional Hopf algebras and the Hennings invariant\footnote{It is sometimes also called Hennings-Kauffman-Radford invariant.} \cite{hen96, kauffman1995invariants} based on finite dimensional ribbon Hopf algebras, both introduced in the 1990's. They are closely related to the Turaev-Viro and Reshetikhin-Turaev construction of quantum invariants from tensor categories. The Kuperberg and the Hennings invariants can be generalized by relaxing the structure of the Hopf algebra or renormalizing the formula. For the development of Hennings-type invariants, see \cite{beliakova2018logarithmic, De_Renzi2018-vx} and references therein, and for that of Kuperberg-type invariants, see \cite{Costantino2020-xe, kashaev2019generalized}. 

Another direction of generalization is to refine the invariants to be sensitive to extra structures of manifolds such as a spin structure or a (co)-homology class. In \cite{sawin2002invariants}, Sawin defined a Henning-type invariant for spin 3-manifolds. In \cite{virel2001, V05}, Virelizier constructed a Hennings-type and a Kuperberg-type invariant for flat bundles of 3-manifolds utilizing a Hopf group-coalgebra. The notion of Hopf group-coalgebras was first introduced by Turaev (\cite{tur00}) because their categories of representations have the natural structure of group-categories, which he used to define invariants of flat bundles on 3-manifolds. More recently in \cite{de2024hennings}, De Renzi, Martel, and Wang used Hopf group-coalgebras to define a Hennings-type 3-dimensional topological quantum field theory (TQFT) for cobordisms decorated with cohomology classes.

In dimension four, Chaidez, Cotler, and Cui (CCC) initiated a Kuperberg-type construction of 4-manifold invariants \cite{CCC22, chaidez2023combed} using  trisection diagrams -- a 3d analog of Heegaard diagrams proposed by Gay and Kirby \cite{GK16}. The CCC invariant takes as input the structure called a Hopf triplet, which consists of three Hopf algebras and three bilinear forms subject to certain compatibility conditions. The invariant recovers some cases of the dichromatic invariant \cite{barenz2018dichromatic} and the Crane-Yetter invariant \cite{crane1997state} of 4-manifolds.

Regarding 4d Hennings-type invariants, De Renzi and Beliakova used Hopf group-coalgebras to define an invariant of 4-dimensional 2-handlebodies decorated with a relative second-cohomology class \cite{Bel23}. Mochida used the foundations set by Virelizier to adapt the Hennings invariant for Kirby diagrams for 4-manifolds with flat principal bundles \cite{Moc24}. Mochida defines the structure of a colored Kirby diagram, which captures the data of the flat bundle, and utilize the structure of  quasitriangular Hopf group-algebras   to define the invariant. This invariant also recovers some cases of  the dichromatic invariant and Crane-Yetter invariant when the flat bundle is trivial.

In this paper, our goal is to extend CCC's invariant of 4-manifolds to an invariant of 4-manifolds with flat bundles. We construct this invariant by generalizing Virilizier's invariant of flat bundles over 3-manifolds to trisection diagrams. The main algebraic structure in this paper will be that of a Hopf $G$-triplet, for a group $G$. This structure is comprised of a Hopf $G$-algebra, two Hopf $G$-coalgebras, and a bilinear form on each pair of them. This structure generalizes Hopf triplets defined by CCC. A particularly interesting construction of a Hopf $G$-triplet is from a quasitriangular Hopf $G$-algebra, which will be used to give a relationship between the invariant defined in this paper and Mochida's invariant. 

The construction of our invariant takes the following steps: Given a group $G$ and a trisection diagram $T$ for a closed orientable connected 4-manifold $X$, we use $T$ to compute a presentation for the fundamental group of $X$ and define a colored trisection diagram colored by a flat $G$-connection, i.e. a homomorphism $\pi_1(X,*)\to G$. The algebraic data is an involutory Hopf $G$-triplet of finite type. With this data, we then use a $G$-analog of the Kuperberg-like construction by CCC to create a tensor network associated to $T$. The contraction of this tensor network with the application of a normalization factor then defines the invariant of flat bundles.

The invariant described in the paper thus generalizes some of the invariants described above. Indeed, when $G$ is trivial, this invariant recovers the 4-manifold invariant defined by CCC \cite{CCC22}, and when the Hopf $G$-triplet comes from a quasitriangular Hopf $G$-algebra, this invariant recovers Mochida's invariant \cite{Moc24}. The structures defined throughout the paper additionally give rise to generalizations of the Kuperberg invariant of 3-manifolds, as well as Virelizier's invariant of flat bundles of 3-manifolds. In function of defining the Hopf $G$-triplet, we define the notion of a Hopf pair of Hopf $G$-coalgebras and the notion of a Hopf $G$-doublet. Applying the construction to a Heegaard diagram using a Hopf pair recovers the Kuperberg invariant \cite{Kup91}, and applying the construction to a colored Heegaard diagram using a Hopf $G$-doublet recovers Virelizier's invariant \cite{V05}.

The paper is organized as follows: In Section \ref{HopfGCoalgebra_Sec}, we review Hopf $G$-(co)algebras and define the structure of a Hopf $G$-triplet. We also give a fundamental lemma of Hopf $G$-triplets which mimics the fundamental lemma of Hopf triplets given by CCC (cf. Lemma 2.19 in \cite{CCC22}). In Section \ref{ColoredTrisectionInvariants_Sec}, we review the definition of trisections and explain how to color trisection diagrams to describe flat bundles. We then introduce the invariant in the form of an invariant of colored trisection diagrams. In Section \ref{FlatBundlesInvariants_Sec}, we define the invariant $Z((\xi,\tilde x);G;\mcal H,e)$, where $(\xi, \tilde x)$ is a pointed flat bundle over a 4-manifold, $G$ is a group, and $(\mcal H,e)$ is a Hopf $G$-triplet along with a set $e$ of integrals. The main theorem in the paper is Theorem \ref{thm:main}: 
\begin{thm*}
    \begin{enumerate}[(a)]
    \item $Z((\xi,\tilde x);G;\mcal H, e)$ is an invariant of pointed flat bundles over 4-manifolds.
    \item The function taking $\tilde x\in \tilde X \mapsto Z((\xi,\tilde x);G;\mcal H,e)$ is constant on the path components of $\tilde X$. 
    \item If $H^\alpha$ is crossed, or $G$ is abelian, or the monodromy of $\xi$ is surjective, then $Z((\xi,\tilde x);G;\mcal H,e)$ does not depend on the base point $\tilde x$.
    \end{enumerate}
\end{thm*} 

Finally, in Sections \ref{MochidaInvariant_Sec}, \ref{3ManifoldInvariants_Sec}, and \ref{Examples_Sec}, we recover Virelizier's and Mochida's invariants described above and give computation showing this invariant can distinguish bundles.

\section{\texorpdfstring{Hopf $G$-Coalgebras}{Hopf G-Coalgebras}}
\label{HopfGCoalgebra_Sec}
\subsection{Tensor Diagrams and Basic Facts}

Let $V$ be a finite dimensional vector space and $V^*$ its linear dual. A \emph{tensor diagram} in $V$ is a pair $(\mcal G, \mcal T = \{\mcal T_v\})$ where,
\begin{enumerate}
    \item $\mcal G$ is a directed graph such that at each vertex $v$, there is a local ordering on the set of incoming legs (i.e. edges) and a local ordering on the set of outgoing legs by $\{1,\dots ,i_v\}$ and $\{1,\dots, o_v\}$, respectively;
    \item for each vertex $v$, $\mcal T_v\in V^1\otimes \cdots \otimes V^{i_v}\otimes V_1\otimes \cdots \otimes V_{o_v}$, where each $V^i$ is a copy of $V^*$ associated with the $i$-th incoming leg and each $V_j$ is a copy of $V$ associated with the $j$-th outgoing leg. In this case, $\mcal T_v$ is called an $(i_v, o_v)$ tensor.
\end{enumerate}

Choosing a basis $\{v_1,\dots ,v_k\}$ for $V$ and a dual basis $\{v^1,\dots ,v^k\}$ for $V^*$, then an $(m,n)$ tensor $\mcal T$ can be generically written as \[\mcal T = \sum \mcal T_{i_1,\dots ,i_m}^{j_1,\dots j_n} v^{i_1}\otimes \cdots \otimes v^{i_m}\otimes v_{j_1}\otimes \cdots \otimes v_{j_n}\] An example of a (3,2) tensor $\mcal T$, as shown tensor-diagrammatically, is:

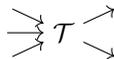
\begin{center}
    \begin{tikzpicture}
        \node at (0,0) (T) {$\mcal T$};
        \foreach \x in {155,180,-155} {
            \draw[<-] (T) -- ++(\x:0.75);
        }
        \foreach \x in {25,-25} {
            \draw[->] (T) -- ++(\x:0.75);
        }
    \end{tikzpicture}
\end{center}

In tensor diagrams, vertices are replaced by the labels of the corresponding tensors. An important convention we use is that incoming edges have local ordering given by enumerating counterclockwise. Similarly, outgoing edges are enumerated clockwise. This convention uniquely determines a local ordering if both types of edges are present. If there is any ambiguity, we will label which edge is first in the ordering.

There is a binary operation on tensors called \emph{contraction}. Given an $(m,n)$ tensor $\mcal T$ and an $(m',n')$ tensor $\mcal T'$ where $n,m'\neq 0$, then we can define a contraction of the $k$th outgoing edge of $\mcal T$ with the $l$th incoming edge of $\mcal T'$ by just connecting the corresponding edges in the diagram.

For example, the figure below shows the contraction of the first outgoing edge of $\mcal T$ with the second incoming edge of $\mcal T'$:

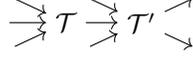
\begin{center}
    \begin{tikzpicture}
        \node at (0,0) (T) {$\mcal T$};
        \node at (1,0) (T') {$\mcal T'$};
        \foreach \x in {155,180,-155} {
            \draw[<-] (T) -- ++(\x:0.75);
        }
        \draw[->] (T) -- (T');
        \foreach \x in {155,-155} {
            \draw[<-] (T') -- ++(\x:0.75);
        }
        \foreach \x in {25,-25} {
            \draw[->] (T') -- ++(\x:0.75);
        }
    \end{tikzpicture}
\end{center}

Because $V$ is finite dimensional, we may identify $(m,n)$ tensors with linear maps from $V^{\otimes m}$ to $V^{\otimes n}$.

\begin{dff}
    Let $\Bbbk$ be a field. A (finite-dimensional) $\Bbbk$-\emph{algebra} is a triple $(H,M,i)$ consisting of a finite-dimensional vector space $H$ over $\Bbbk$, a (2,1) tensor $M$ called the \emph{multiplication}, and a (0,1) tensor $i$ called the \emph{unit} satisfying:

    \noindent Axiom 1: (Associativity)
    \begin{center}
        \begin{tikzpicture}
            \node at (0,0) (ML) {$M$};
            \node at (1,0) (MR) {$M$};
            \draw[->] (ML) -- (MR);
            \draw[<-] (ML) -- ++(135:0.75);
            \draw[<-] (ML) -- ++(-135:0.75);
            \draw[<-] (MR) -- ++(-135:0.75);
            \draw[->] (MR) -- ++(0.75,0);
            \node at (2.5,0) {$=$};
            \begin{scope}[shift = {(3.75,0)}]
                \node at (0,0) (ML) {$M$};
                \node at (1,0) (MR) {$M$};
                \draw[->] (ML) -- (MR);
                \draw[<-] (ML) -- ++(135:0.75);
                \draw[<-] (ML) -- ++(-135:0.75);
                \draw[<-] (MR) -- ++(135:0.75);
                \draw[->] (MR) -- ++(0.75,0);
            \end{scope}
        \end{tikzpicture}
    \end{center}

    \noindent Axiom 2: (Unit)
    \begin{center}
        \begin{tikzpicture}
            \node at (1,0) (M) {$M$};
            \node[above left = 0.25 of M] (i) {$i$};
            \draw[->] (i) -- (M);
            \draw[<-] (M) -- ++(-135:0.75);
            \draw[->] (M) -- ++(0.75,0);
            \node at (2.5,0) {$=$};
            \begin{scope}[shift = {(3.25,0)}]
                \node at (1,0) (M) {$M$};
                \node[below left = 0.25 of M] (i) {$i$};
                \draw[->] (i) -- (M);
                \draw[<-] (M) -- ++(135:0.75);
                \draw[->] (M) -- ++(0.75,0);
            \node at (2.5,0) {$=$};
            \begin{scope}[shift = {(3.25,0)}]
                \draw[->] (0,0) -- (1,0);
            \end{scope}
            \end{scope}
        \end{tikzpicture}
    \end{center}
\end{dff}

\begin{dff}
    Let $G$ be a group with identity element $1$. A \emph{Hopf G-coalgebra} (over $\Bbbk$) consists of a family of $\Bbbk$-algebras $H = \{ (H_g, M_g, i_g) \}_{g\in G}$ endowed with a family $\Delta = \{\Delta_{g,h}:H_{gh}\to H_g\otimes H_h\}_{g,h\in G}$ of algebra morphisms called the \emph{comultiplicaton}, an algebra morphism $\epsilon:H_1\to \Bbbk$ called the \emph{counit}, and a family $S = \{S_g :H_g\to H_{g^{-1}}\}_{g\in G}$ of $\Bbbk$-linear maps called the \emph{antipode} such that, for all $g,h,k\in G$, the following axioms are satisfied:

    \begin{enumerate}[{Axiom} 1:]
    \item (Coassociativity)
    \begin{center}
        \begin{tikzpicture}
            \node at (0,0) (DL) {$\Delta_{gh,k}$};
            \node at (1.5,0) (DR) {$\Delta_{g,h}$};
            \draw[<-] (DL) -- ++(-1,0);
            \draw[->] (DL) -- (DR);
            \draw[->] (DL) -- ++(-45:0.75);
            \foreach \x in {45,-45} {
                \draw[->] (DR) -- ++(\x:0.75);
            }
            \node at (3,0) {$=$};
            \begin{scope}[shift = {(5,0)}]
                \node at (0,0) (DL) {$\Delta_{g,hk}$};
                \node at (1.5,0) (DR) {$\Delta_{h,k}$};
                \draw[<-] (DL) -- ++(-1,0);
                \draw[->] (DL) -- (DR);
                \draw[->] (DL) -- ++(45:0.75);
                \foreach \x in {45,-45} {
                    \draw[->] (DR) -- ++(\x:0.75);
                }
            \end{scope}
        \end{tikzpicture}
    \end{center}

    \item (Counit)
    \begin{center}
        \begin{tikzpicture}
            \node at (0,0) (D) {$\Delta_{1,g}$};
            \node[above right = 0.25 of D] (e) {$\epsilon$};
            \draw[<-] (D) -- ++(-1,0);
            \draw[->] (D) -- ++(-45:0.75);
            \draw[->] (D) -- (e);
            \node at (1.6,0) {$=$};
            \begin{scope}[shift = {(3.2,0)}]
            \node at (0,0) (D) {$\Delta_{g,1}$};
            \node[below right = 0.25 of D] (e) {$\epsilon$};
            \draw[<-] (D) -- ++(-1,0);
            \draw[->] (D) -- ++(45:0.75);
            \draw[->] (D) -- (e);
            \end{scope}
            \node at (4.8,0) {$=$};
            \draw[->] (5.7,0) -- (6.7,0);
        \end{tikzpicture}
    \end{center}
    
    \item (Antipode)
\begin{center}
        \begin{tikzpicture}
            \node at (0,0) (D1) {$\Delta_{g^{-1},g}$};
            \node at (1.5,0.5) (S1) {$S_{g^{-1}}$};
            \node at (2.75,0) (M1) {$M_g$};
            \draw[<-] (D1) -- ++(-1,0);
            \draw[->] (D1) -- (S1);
            \draw[->] (S1) -- (M1);
            \draw[->] (D1) .. controls +(1,-0.5) and +(-1,-0.5) .. (M1);
            \draw[->] (M1) -- ++(1,0);
            \node at (4.5,0) {$=$};
            \begin{scope}[shift = {(6.3,0)}]
                \node at (0,0) (D1) {$\Delta_{g,g^{-1}}$};
                \node at (1.5,-0.5) (S1) {$S_{g^{-1}}$};
                \node at (2.75,0) (M1) {$M_g$};
                \draw[<-] (D1) -- ++(-1,0);
                \draw[->] (D1) -- (S1);
                \draw[->] (S1) -- (M1);
                \draw[->] (D1) .. controls +(1,0.5) and +(-1,0.5) .. (M1);
                \draw[->] (M1) -- ++(1,0);
            \end{scope}
            \begin{scope}[shift = {(6,-1.25)}]
                \node at (-1.5,0) {$=$};
                \node at (0,0) (e) {$\epsilon$};
                \node at (1,0) (i) {$i_g$};
                \draw[<-] (e) -- ++(-0.75,0);
                \draw[->] (i) -- ++(0.75,0);
            \end{scope}
        \end{tikzpicture}
    \end{center}
    \end{enumerate}

    $H$ is \emph{involutory} if 
    \begin{center}
    \begin{tikzpicture}
        \node at (0,0) (S1) {$S_g$};
        \node at (1,0) (S2) {$S_{g^{-1}}$};
        \draw[<-] (S1) -- ++(-0.75,0);
        \draw[->] (S1) -- (S2);
        \draw[->] (S2) -- ++(0.75,0);
        \node at (2.5,0) {$=$};
        \draw[->] (3.25,0) -- (4.25,0);
    \end{tikzpicture}
    \end{center}
\end{dff}

Note that when $G = 1$, then $H$ is a Hopf algebra. In particular, restricting to the 1-component yields a Hopf algebra $(H_1,M_1, i_1, \Delta_{1,1},\epsilon, S_1)$.

We say that $H$ is of \emph{finite type} if $H_g$ is finite dimensional for all $g\in G$. Note that the definition above is a definition for a Hopf $G$-coalgebra of finite type. We say that $H$ is \emph{finite-dimensional} if $H$ is finite-dimensional as a graded vector space.

We will use the following abbreviations (in light of (co)associativity)

\begin{center}
    \begin{tikzpicture}
        \node at (0,0) (M0) {$M_g$};
        \draw[->] (M0) -- ++(0.75,0);
        \foreach \x in {90,120,150,240}
        {
            \draw[<-] (M0) -- ++(\x:0.75);
        }
        \foreach \x in {175,195,215}
        {
            \filldraw (\x:0.55cm) circle [radius = 0.5pt];
        }
        \node at (1.5,0) {$=$};
    \begin{scope}[shift = {(3,0)}]
        \node at (0,0) (M1) {$M_g$};
        \node at (1,0) (M2) {$M_g$};
        \node at (2,0) (dots) {$\cdots$};
        \node at (3,0) (M3) {$M_g$};
        \draw[<-] (M1) -- ++(-0.75,0);
        \draw[<-] (M1) -- ++(-135:0.75);
        \draw[<-] (M2) -- ++(-135:0.75);
        \draw[<-] (M3) -- ++(-135:0.75);
        \draw[->] (M1) -- (M2);
        \draw[->] (M2) -- (dots);
        \draw[->] (dots) -- (M3);
        \draw[->] (M3) -- ++(0.75,0);
    \end{scope}
    \end{tikzpicture}\vspace{0.5cm}
    
    \begin{tikzpicture}
        \node at (0,0) (M) {$M_g$};
        \draw[<-] (M) -- ++(-0.75,0);
        \draw[->] (M) -- ++(0.75,0);
        \node at (1.25,0) {$=$};
        \draw[->] (2,0) -- (2.6,0);
    \begin{scope}[shift = {(4,0)}]
        \node at (0,0) (M1) {$M_g$};
        \draw[->] (M1) -- ++(0.75,0);
        \node at (1.25,0) {$=$};
        \node at (2,0) (i) {$i_g$};
        \draw[->] (i) -- ++(0.75,0);
    \end{scope}
    \end{tikzpicture}

    \begin{tikzpicture}
        \node at (0,0) (D) {$\Delta_{g_1,\dots, g_n}$};
        \draw[<-] (D) -- ++(-1.25,0);
        \foreach \angle in {90,30,60,-60}
        {
            \draw[->] (D) -- ++(\angle:1);
        }
        \foreach \angle in {5,-15,-35}
        {
            \filldraw (\angle:0.75cm) circle [radius = 0.5pt];
        }
        \node at (1.75,0) {$=$};
    \begin{scope}[shift = {(4,0)}]
        \node at (0,0) (D1) {$\Delta_{g_1g_2\cdots g_{n-1},g_n}$};
        \node at (1.75,0) (dots) {$\cdots$};
        \node at (3,0) (D2) {$\Delta_{g_1g_2,g_3}$};
        \node at (4.5,0) (D3) {$\Delta_{g_1,g_2}$};
        \draw[<-] (D1) -- ++(-1.5,0);
        \draw[->] (D1) -- ++(-45:0.75);
        \draw[->] (D2) -- ++(-45:0.75);
        \draw[->] (D3) -- ++(-45:0.75);
        \draw[->] (D1) -- (dots);
        \draw[->] (dots) -- (D2);
        \draw[->] (D2) -- (D3);
        \draw[->] (D3) -- ++(1,0);
    \end{scope}
    \end{tikzpicture}\vspace{0.25cm}
    
    \begin{tikzpicture}
        \node at (0,0) (D) {$\Delta_g$};
        \draw[<-] (D) -- ++(-0.75,0);
        \draw[->] (D) -- ++(0.75,0);
        \node at (1.25,0) {$=$};
        \draw[->] (2,0) -- (2.6,0);
    \begin{scope}[shift = {(4,0)}]
        \node at (0.5,0) (D1) {$\Delta_1$};
        \draw[<-] (D1) -- ++(-0.75,0);
        \node at (1.25,0) {$=$};
        \node at (2.5,0) (e) {$\epsilon$};
        \draw[<-] (e) -- ++(-0.75,0);
    \end{scope}
    \end{tikzpicture}

    \end{center}

If $H$ is a Hopf $G$-coalgebra with invertible antipode (i.e. each $S_g$ is invertible), then $H^{op}$ is a new Hopf $G$-coalgebra with $(H^{op})_g = H_g$, and multiplication $M^{op} = \{M_g^{op}\}_{g\in G}$, where $M_g^{op}$ is the tensor

\begin{center}
    \begin{tikzpicture}
        \node at (0,0) (M) {$M_g^{op}$};
        \foreach \angle in {150,210} {
        \draw[<-] (M) -- ++(\angle:1);
        }
        \draw[->] (M) -- ++(0.75,0);
        \node at (1.5,0) {$=$};
        \begin{scope}[shift = {(3.25,0)}]
            \node at (0.5,0) (M1) {$M_g$};
        \draw[->] (-1,-0.25) .. controls +(0.5,0) and +(-1,0.5) .. (M1);
        \draw[->] (-1,0.25) .. controls +(0.5,0) and +(-1,-0.5) .. (M1);
        \draw[->] (M1) -- ++(0.75,0);
        \end{scope}
    \end{tikzpicture}
\end{center}

The antipode is given by $S^{op} = \{S_g^{op} = (S_{g^{-1}})^{-1}\}_{g\in G}$, and all other tensors are kept the same. Similarly, we can define $H^{cop}$, which is defined by setting $(H^{cop})_g = H_{g^{-1}}$ with the comultiplication $\Delta_{g,h}^{cop}$ given by the tensor

\begin{center}
    \begin{tikzpicture}
        \node at (0,0) (D) {$\Delta_{g,h}^{cop}$};
        \foreach \angle in {30,-30} {
        \draw[->] (D) -- ++(\angle:1);
        }
        \draw[<-] (D) -- ++(-1,0);
        \node at (1.5,0) {$=$};
        \begin{scope}[shift = {(3,0)}]
            \node at (0.5,0) (D1) {$\Delta_{h^{-1},g^{-1}}$};
        \draw[->] (D1) .. controls +(1,0.5) and +(-0.5,0) .. (2.5,-0.25);
        \draw[->] (D1) .. controls +(1,-0.5) and +(-0.5,0) .. (2.5,0.25);
        \draw[<-] (D1) -- ++(-1.25,0);
        \end{scope}
    \end{tikzpicture}
\end{center}

Notice that $\Delta_{g,h}^{cop}$ takes in an element of $H^{cop}_{gh} = H_{h^{-1}g^{-1}}$, and outputs an element in $H^{cop}_g \otimes H^{cop}_h = H_{g^{-1}} \otimes H_{h^{-1}}$, hence the labels are consistent with the conventions for Hopf $G$-coalgebras.

The antipode is given by $S^{cop} = \{S_g^{cop} = (S_{g})^{-1}\}_{g\in G}$, and all other tensors are kept the same.

In a Hopf $G$-coalgebra, the antipode $S$ is an anti-algebra morphism and satisfies:

\begin{center}
    \begin{tikzpicture}
        \node at (0,0) (S) {$S_{gh}$};
        \node at (1.5,0) (D) {$\Delta_{h^{-1},g^{-1}}$};
        \draw[<-] (S) -- ++(-0.75,0);
        \draw[->] (S) -- (D);
        \draw[->] (D) -- ++(30:1);
        \draw[->] (D) -- ++(-30:1);
        \node at (3.5,0) {$=$};
        \begin{scope}[shift = {(5.5,0)}]
            \node at (0,0) (D1) {$\Delta_{h^{-1},g^{-1}}^{cop}$};
            \node at (1.5,0.5) (SU) {$S_h$};
            \node at (1.5,-0.5) (SL) {$S_g$};
            \draw[<-] (D1) -- ++(-1.25,0);
            \draw[->] (D1) -- (SU);
            \draw[->] (D1) -- (SL);
            \draw[->] (SU) -- ++(0.75,0);
            \draw[->] (SL) -- ++(0.75,0);
        \end{scope}
    \end{tikzpicture}
\end{center}

\begin{dff}
    A \emph{morphism} of Hopf $G$-coalgebras from $H$ to $H'$ is a family of algebra morphisms $f_g: H_g \to H_g'$ such that for all $g,h\in G$ \[(f_g\otimes f_h)\circ \Delta_{g,h} = \Delta'_{g,h} \circ f_{gh}\] \[\epsilon\circ f_1 = \epsilon'\] and \[S_g'\circ f_g = f_{g^{-1}}\circ S_g\]
\end{dff}

\begin{rmk}
    We note that if $G$ is trivial, then a morphism of Hopf $G$-coalgebras is a morphism of Hopf algebras.
\end{rmk}

\begin{lem}
    Let $H$ be an involutory Hopf $G$-coalgebra of finite type. Then the following \emph{ladders} are invertible:

    \begin{center}
        \begin{tikzpicture}
            \node at (0,0) (M1) {$M_g$};
            \node at (0,-2) (D1) {$\Delta_{g,h}$};
            \draw[<-] (M1) -- ++(-0.75,0);
            \draw[<-] (D1) -- ++(-0.75,0);
            \draw[->] (D1) -- (M1);
            \draw[->] (M1) -- ++(0.75,0);
            \draw[->] (D1) -- ++(0.75,0);
        \begin{scope}[shift = {(3,0)}]
            \node at (0,0) (M2) {$M_g$};
            \node at (0,-2) (D2) {$\Delta_{g^{-1},h}$};
            \node at (0,-1) (S2) {$S_{g^{-1}}$};
            \draw[<-] (M2) -- ++(-1,0);
            \draw[<-] (D2) -- ++(-1,0);
            \draw[->] (M2) -- ++(1,0);
            \draw[->] (D2) -- ++(1,0);
            \draw[->] (D2) -- (S2);
            \draw[->] (S2) -- (M2);
        \end{scope}
        \begin{scope}[shift = {(6,0)}]
            \node at (0,0) (M3) {$M_g$};
            \node at (0,-2) (D3) {$\Delta_{h,g}$};
            \draw[<-] (M3) -- ++(-0.75,0);
            \draw[->] (D3) -- ++(-0.75,0);
            \draw[->] (D3) -- (M3);
            \draw[->] (M3) -- ++(0.75,0);
            \draw[<-] (D3) -- ++(0.75,0);
        \end{scope}
        \begin{scope}[shift = {(9,0)}]
            \node at (0,0) (M4) {$M_g$};
            \node at (0,-2) (D4) {$\Delta_{h,g^{-1}}$};
            \node at (0,-1) (S4) {$S_{g^{-1}}$};
            \draw[<-] (M4) -- ++(-1,0);
            \draw[->] (D4) -- ++(-1,0);
            \draw[->] (M4) -- ++(1,0);
            \draw[<-] (D4) -- ++(1,0);
            \draw[->] (D4) -- (S4);
            \draw[->] (S4) -- (M4);
        \end{scope}
        \end{tikzpicture}\vspace{0.5cm}

        \begin{tikzpicture}
            \node at (0,0) (M1) {$M_g$};
            \node at (0,-2) (D1) {$\Delta_{g,h}$};
            \draw[->] (M1) -- ++(-0.75,0);
            \draw[<-] (D1) -- ++(-0.75,0);
            \draw[->] (D1) -- (M1);
            \draw[<-] (M1) -- ++(0.75,0);
            \draw[->] (D1) -- ++(0.75,0);
        \begin{scope}[shift = {(3,0)}]
            \node at (0,0) (M2) {$M_g$};
            \node at (0,-2) (D2) {$\Delta_{g^{-1},h}$};
            \node at (0,-1) (S2) {$S_{g^{-1}}$};
            \draw[->] (M2) -- ++(-1,0);
            \draw[<-] (D2) -- ++(-1,0);
            \draw[<-] (M2) -- ++(1,0);
            \draw[->] (D2) -- ++(1,0);
            \draw[->] (D2) -- (S2);
            \draw[->] (S2) -- (M2);
        \end{scope}
        \begin{scope}[shift = {(6,0)}]
            \node at (0,0) (M3) {$M_g$};
            \node at (0,-2) (D3) {$\Delta_{h,g}$};
            \draw[->] (M3) -- ++(-0.75,0);
            \draw[->] (D3) -- ++(-0.75,0);
            \draw[->] (D3) -- (M3);
            \draw[<-] (M3) -- ++(0.75,0);
            \draw[<-] (D3) -- ++(0.75,0);
        \end{scope}
        \begin{scope}[shift = {(9,0)}]
            \node at (0,0) (M4) {$M_g$};
            \node at (0,-2) (D4) {$\Delta_{h,g^{-1}}$};
            \node at (0,-1) (S4) {$S_{g^{-1}}$};
            \draw[->] (M4) -- ++(-1,0);
            \draw[->] (D4) -- ++(-1,0);
            \draw[<-] (M4) -- ++(1,0);
            \draw[<-] (D4) -- ++(1,0);
            \draw[->] (D4) -- (S4);
            \draw[->] (S4) -- (M4);
        \end{scope}
        \end{tikzpicture}
    \end{center}
\end{lem}


\begin{proof} Here is an example; the rest are left up to the reader.
\begin{center}
    \begin{tikzpicture}
            \node at (0,0) {$M_g$};
            \node at (0,-2) {$\Delta_{g,h}$};
            \draw[->] (0,-1.7) -- (0,-0.3);
            \draw[->] (-0.8,0) -- (-0.4,0);
            \draw[->] (0.4,0) -- (1.1,0);
            \draw[->] (-1,-2) -- (-0.6,-2);
        \begin{scope}[shift = {(1.5,0)}]
            \node at (0,0) {$M_g$};
            \node at (0.1,-2) {$\Delta_{g^{-1},gh}$};
            \draw[->] (0.4,0) -- (0.8,0);
            \draw[->] (-1,-2) -- (-0.6,-2);
            \draw[->] (0.8,-2) -- (1.2,-2);
            \node at (0,-1) {$S_{g^{-1}}$};
            \draw[->] (0,-1.7) -- (0,-1.3);
            \draw[->] (0,-0.7) -- (0,-0.3);
            \node at (2,-1) {$=$};
        \end{scope}
        \begin{scope}[shift = {(7.5,0)}]
            \node at (0.4,0) {$M_g$};
            \node at (-2,-2) {$\Delta_{1,h}$};
            \draw[->] (-2.8,-2) -- (-2.4,-2);
            \draw[->] (-2.8,0) -- (-0.1,0);
            \draw[->] (-1.6,-2) -- (1.2,-2);
            \draw[->] (0.8,0) -- (1.2,0);
        \begin{scope}[shift = {(-1.7,-1)}]
            \node at (-0.3,0) {$\Delta_{g,g^{-1}}$};
            \node at (1,-0.5) {$S_{g^{-1}}$};
            \node at (2.1,0) {$M_g$};
            \draw[->] (-0.3,-0.7) -- (-0.3,-0.3);
            \draw[->] (0.3,-0.1) -- (0.6,-0.3);
            \draw[->] (1.35,-0.35) -- (1.7,-0.1);
            \draw[->] (0.3,0.1) .. controls +(0.25,0.25) and +(-0.25,0.25) .. (1.7,0.1);
            \draw[->] (2.1,0.3) -- (2.1,0.7);
        \end{scope}
        \end{scope}
        \end{tikzpicture}\vspace{0.5cm}

        \begin{tikzpicture}
        \node at (2.4,-1) {$=$};
        \begin{scope}[shift = {(3.5,0)}]
            \draw[->] (0,0) -- (1,0);
            \draw[->] (0,-2) -- (1,-2);
        \end{scope}
    \end{tikzpicture}
\end{center}

\end{proof}

\begin{dff}
    A \emph{left $G$-integral} in a Hopf $G$-coalgebra $H$ is a family of $\Bbbk$-linear forms $\mu_L = \{\mu_{L,g}\}_{g\in G}$, $\mu_{L,g}\in H_g^*$ such that for any $g,h\in G$,
    \begin{center}
        \begin{tikzpicture}
            \node at (-0.2,0) {$\Delta_{g,h}$};
                \foreach \angle in {30,-30}
                {
                    \draw[->] (\angle:0.3) -- (\angle:0.6);
                }
                \draw[->] (-1,0) -- (-0.7,0);
                \node at (1,-0.4) {$\mu_{L,h}$};
                \node at (1.6,0) {$=$};
                \draw[->] (2.3,0) -- (2.6,0);
                \node at (3.2,0) {$\mu_{L,gh}$};
                \node at (4.1,0) {$i_g$};\
                \draw[->] (4.4,0) -- (4.7,0);
        \end{tikzpicture}
    \end{center}

    Similarly one can define a \emph{right $G$-integral}. Since $H_1$ forms a Hopf algebra, we can define \emph{left} (resp. right) \emph{cointegrals}, which are elements $e_L\in H_1$ 
    \begin{center}
        \begin{tikzpicture}
            \node at (0.1,0) {$M_1$};
                \foreach \angle in {150,210}
                {
                    \draw[->] (\angle:0.6) -- (\angle:0.3);
                }
                \draw[->] (0.5,0) -- (0.8,0);
                \node at (-0.9,-0.4) {$e_L$};
                \node at (1.4,0) {$=$};
                \draw[->] (2,0) -- (2.3,0);
                \node at (2.6,0) {$\epsilon$};
                \node at (3.1,0) {$e_L$};\
                \draw[->] (3.4,0) -- (3.7,0);
        \end{tikzpicture}
    \end{center}
\end{dff}


We will say that a left (resp. right) $G$-integral for $H$ is \emph{nonzero} if $\mu_{L,g}$ is nonzero for some $g\in G$. By \cite[Lemma 3.1]{V02}, any non-zero left (resp. right) $G$-integral has $\mu_{L,1} \neq 0$, and hence $\mu_{L,g}\neq 0$ for all $g$ with $H_g \neq 0$. The space of left (resp. right) $G$-integrals in a finite-dimensional Hopf $G$-coalgebra is one-dimensional (\cite[Theorem 3.6]{V02}) and it is clear that for an involutory Hopf $G$-algebra, every right $G$-integral is also a left $G$-integral.

A Hopf $G$-coalgebra is said to be \emph{semisimple} if each algebra $H_g$ is semisimple. When $H$ is of finite type, then $H$ is semisimple if and only if $H_1$ is semisimple (\cite[Lemma 5.1]{V02}).

\begin{dff}
    A \emph{right $G$-comodule} over a Hopf $G$-coalgebra $H$ is a family $R = \{R_g\}_{g\in G}$ of vector spaces endowed with a family $\Delta^R = \{\Delta^R_{g,h}:R_{gh}\to R_g\otimes H_h\}_{g,h\in G}$ of $\Bbbk$-linear maps such that
    \begin{center}
        \begin{tikzpicture}
            \node at (0.2,0) {$\Delta^R_{g,h}$};
            \node at (1.5,0) {$\Delta^R_{gh,k}$};
            \draw[<-] (0.6,0) -- (1,0);
            \draw[<-] (-0.5,0.5) -- (-0.2,0.2);
            \draw[<-] (-0.5,-0.5) -- (-0.2,-0.2);
            \draw[<-] (2,0) -- (2.4,0);
            \draw[<-] (0.8,0.5) -- (1.1,0.2);
            \node at (3,0) {$=$};
            \begin{scope}[shift = {(4.5,0)}]
            \node at (0.2,0) {$\Delta_{h,k}$};
            \node at (1.5,0) {$\Delta^R_{g,hk}$};
            \draw[<-] (0.6,0) -- (1,0);
            \draw[<-] (-0.5,0.5) -- (-0.2,0.2);
            \draw[<-] (-0.5,-0.5) -- (-0.2,-0.2);
            \draw[<-] (2,0) -- (2.4,0);
            \draw[<-] (0.8,-0.5) -- (1.1,-0.2);
            \end{scope}
        \end{tikzpicture}
    \end{center}
    \begin{center}
        \begin{tikzpicture}
            \node at (0.2,0) {$\Delta^R_{g,1}$};
            \node at (-0.7,0.7) {$\epsilon$};
            \draw[<-] (0.6,0) -- (1,0);
            \draw[<-] (-0.5,0.5) -- (-0.2,0.2);
            \draw[<-] (-0.5,-0.5) -- (-0.2,-0.2);
            \node at (1.6,0) {$=$};
            \draw[<-] (2.2,0) -- (3.2,0);
        \end{tikzpicture}
    \end{center}

    A \emph{G-subcomodule} of $R$ is a family $Q = \{Q_g\}_{g\in G}$ where $Q_g$ is a subspace of $R_g$ and $\Delta^R_{g,h}(Q_{gh})\subset Q_g\otimes H_h$ for all $g,h\in G$.

    We say that a right $G$-comodule $R$ is \emph{simple} if it is nonzero and has no $G$-subcomodules other than 0 and itself. $R$ is \emph{cosemisimple} if it is a direct sum of simple right $G$-comodules.
\end{dff}

A Hopf $G$-coalgebra $H$ is called \emph{cosemisimple} if it is cosemisimple as a right $G$-comodule over itself with $\Delta^H_{g,h} = \Delta_{g,h}$. By \cite[Theorem 5.4]{V02}, $H$ is cosemisimple if and only if there exists a right $G$-integral $\mu$ with $\mu_{g}(i_g) = 1$ for all $g\in G$ for which $H_g\neq 0$. As a corollary, if $H$ is of finite type, then $H$ is cosemisimple if and only if $H_1$ is cosemisimple.

 In fact, by \cite[Corollary 5.6]{V02}, if $\dim(H_1)\neq 0$, then a finite type involutory Hopf $G$-coalgebra is semisimple and cosemisimple. In particular, $\dim(H_1)\neq 0$ if $\Bbbk$ is characteristic 0.

Let $\mu$ be a nonzero two-sided $G$-integral in a finite type involutory Hopf $G$-coalgebra, and let $e$ be a nonzero cointegral in $H_1$ such that $\mu_1(e) = 1$, then we will use the following abbreviations.

\begin{center}
\begin{tikzpicture}
            \node at (0.1,0) {$M_g$};
                \foreach \angle in {90,120,150,240}
                {
                    \draw[->] (\angle:0.6cm) -- (\angle:0.3cm);
                }
                \foreach \angle in {175,195,215}
                {
                    \filldraw (\angle:0.45cm) circle [radius = 0.5pt];
                }
            \node at (0.9,0) {$=$};
            \begin{scope}[shift = {(2,0)}]
            \node at (0.1,0) {$M_g$};
            \draw[->] (0.5,0) -- (0.8,0);
            \foreach \angle in {90,120,150,240}
            {
                \draw[->] (\angle:0.6cm) -- (\angle:0.3cm);
            }
            \foreach \angle in {175,195,215}
            {
                \filldraw (\angle:0.45cm) circle [radius = 0.5pt];
            }
            \node at (1.1,0) {$\mu_g$};
                
            \end{scope}
    \begin{scope}[shift = {(6,0)}]
            \node at (0.5,0) {$\Delta_{g_1,\dots ,g_n}$};
                \foreach \angle in {90,120,150,240}
                {
                    \draw[<-] (\angle:0.6cm) -- (\angle:0.3cm);
                }
                \foreach \angle in {175,195,215}
                {
                    \filldraw (\angle:0.45cm) circle [radius = 0.5pt];
                }
            \node at (1.8,0) {$=$};
            \begin{scope}[shift = {(3,0)}]
                \node at (0.5,0) {$\Delta_{g_1,\dots,g_n}$};
            \draw[<-] (1.3,0) -- (1.6,0);
            \foreach \angle in {90,120,150,240}
            {
                \draw[<-] (\angle:0.6cm) -- (\angle:0.3cm);
            }
            \foreach \angle in {175,195,215}
            {
                \filldraw (\angle:0.45cm) circle [radius = 0.5pt];
            }
            \node at (1.8,0) {$e$};
            \end{scope}
        \end{scope}
    \end{tikzpicture}
\end{center}

whenever $g_1\cdots g_n = 1$.

In fact, $\mu_g$ and $e$ are cyclic, which justifies these abbreviations.

\begin{lem}
    Let $H$ be an involutory Hopf $G$-coalgebra and let $\mu$ and $e$ be a nonzero $G$-integral and a nonzero cointegral, respectively, then
    \begin{center}
        \begin{tikzpicture}
            \node at (-2.4,0) {$e$};
            \draw[->] (-2.2,0) -- (-1.8,0);
            \node at (-1.4,0) {$\mu_1$};
            \draw[->] (-0.6,0) -- (-0.3,0);
            \node at (0,0) {$S_g$};
            \draw[->] (0.3,0) -- (0.6,0);
            \node at (1.4,0) {$=$};
            \begin{scope}[shift = {(2.6,0)}]
                \draw[->] (-0.1,0) -- (0.3,0);
                \node at (-0.3,0) {$e$};
                \node at (1,0) {$\Delta_{g,g^{-1}}$};
                \draw[->] (1.6,0) -- (2,0);
                \node at (2.4,0) {$M_g$};
                \draw[->] (2.8,0) -- (3.2,0);
                \node at (3.6,0) {$\mu_g$};
                \draw[->] (1.2,-1) -- (2.1,-0.2);
                \draw[->] (1.2,-0.2) -- (2.1,-1);
            \end{scope}
        \end{tikzpicture}
    \end{center}
\end{lem}

\begin{proof}
    We begin by putting a ladder on the right-hand side:

    \begin{center}
        \begin{tikzpicture}
                \draw[->] (-0.1,0) -- (0.3,0);
                \node at (-0.3,0) {$e$};
                \node at (1,0) {$\Delta_{g,g^{-1}}$};
                \draw[->] (1.6,0) -- (2,0);
                \node at (2.4,0) {$M_g$};
                \draw[->] (2.8,0) -- (3.2,0);
                \node at (3.6,0) {$\mu_g$};
                \draw[->] (1.2,-1) -- (2.1,-0.2);
                \draw[->] (1.2,-0.2) -- (2.1,-1);
                \node at (1,-1.2) {$\Delta_{g,g^{-1}}$};
                \draw[->] (1.6,-1.2) -- (2,-1.2);
                \node at (2.6,-1.2) {$M_{g^{-1}}$};
                \draw[->] (-0.1,-1.2) -- (0.3,-1.2);
                \draw[->] (3.2,-1.2) -- (3.6,-1.2);
                \node at (4.7,-0.6) {$=$};
                \begin{scope}[shift = {(6,0)}]
                    \draw[->] (-0.1,0) -- (0.6,-0.4);
                    \draw[->] (-0.1,-1.2) -- (0.6,-0.8);
                \node at (-0.3,0) {$e$};
                \node at (2.4,-0.6) {$\Delta_{g,g^{-1}}$};
                \draw[->] (1.4,-0.6) -- (1.8,-0.6);
                \node at (1,-0.6) {$M_1$};
                \draw[->] (2.7,-0.4) -- (3.2,0);
                \node at (3.6,0) {$\mu_g$};
                \draw[->] (2.7,-0.8) -- (3.2,-1.2);
                \end{scope}
                \end{tikzpicture}\vspace{0.5cm}
                \begin{tikzpicture}
                \node at (4.5,-0.6) {$=$};
                \node at (6,0) {$e$};
                \draw[->] (6.2,0) -- (6.6,0);
                \node at (7,0) {$\mu_1$};
                \node at (6,-1.2) {$\epsilon$};
                \draw[->] (5.4,-1.2) -- (5.8,-1.2);
                \node at (7,-1.2) {$i_{g^{-1}}$};
                \draw[->] (7.3,-1.2) -- (7.7,-1.2);
        \end{tikzpicture}
    \end{center}

    Applying the inverse ladder then yields the result.
\end{proof}

\begin{cor}\label{trace_antipode_inv}
    $\mu_{g^{-1}}\circ S_g = \mu_g$ and  $S_1\circ e = e$
\end{cor}

Dual to the notion of Hopf $G$-coalgebras is a Hopf $G$-algebra.
\begin{dff}
     A \emph{Hopf $G$-algebra} $H$ (over $\Bbbk$) is a family of $\Bbbk$-coalgebras $\{(H_g, \Delta_g, \epsilon_g)\}_{g\in G}$ endowed with a family $\{M_{g,h}:H_g\otimes H_h\to H_{gh}\}_{g,h\in G}$ of coalgebra morphisms called the \emph{multiplication}, a coalgebra morphism $i:\Bbbk\to H_1$ called the \emph{unit}, and a family $\{S_g:H_g\to H_{g^{-1}}\}_{g\in G}$ of $\Bbbk$-linear maps called the \emph{antipode} such that it satisfies the dual variants of the axioms of a Hopf $G$-coalgebra. Namely, the associativity, unit, and antipode axioms.

     A \emph{morphism} of Hopf $G$-algebras is a family of coalgebra morphisms which intertwines the multiplication, unit, and antipode.
\end{dff}

One can define left and right $G$-cointegrals for Hopf $G$-algebras analogously to left and right $G$-integrals for Hopf $G$-coalgebras. They satisfy dual versions of the results above.

There is a duality between Hopf $G$-algebras and Hopf $G$-coalgebras. Namely, if $H$ is a Hopf $G$-algebra, then we can obtain a Hopf $G$-coalgebra $\tilde{H}$ in the following way: set $\tilde{H}_g = (H_g)^*$, and let the structure tensors be \[\tilde{M}_g = (\Delta_g)^*,\qquad \tilde{i}_g = (\epsilon_g)^*\] \[\tilde{\Delta}_{g,h} = (M_{g,h})^*, \qquad \tilde{\epsilon} = i^*\] \[\tilde{S}_g = (S_g)^*\] We denote $\tilde H$ by $H^*$. We can similarly obtain a Hopf $G$-algebra from taking the dual of a Hopf $G$-coalgebra.

\subsection{\texorpdfstring{Crossed Hopf $G$-Coalgebras}{Crossed Hopf G-Coalgebras}}

\begin{dff}
    A Hopf $G$-coalgebra is said to be \emph{crossed} if it is endowed with a family $\phi = \{\phi_h: H_g\to H_{hgh^{-1}}\}_{g,h\in G}$  of algebra isomorphisms called the \emph{crossing} such that 
    \begin{center}
        \begin{tikzpicture}
            \node at (-0.1,0) {$\Delta_{g,k}$};
            \draw[->] (30:0.4) -- (30:0.8);
            \draw[->] (-30:0.4) -- (-30:0.8);
            \draw[->] (-0.9,0) -- (-0.5,0);
            \node at (1,0.5) {$\phi_h$};
            \node at (1,-0.5) {$\phi_h$};
            \draw[->] (1.4,0.5) -- (1.8,0.5);
            \draw[->] (1.4,-0.5) -- (1.8,-0.5);
            \node at (2.5,0) {$=$};
            \begin{scope}[shift ={(4,0)}]
                \node at (0,0) {$\phi_h$};
                \draw[->] (-0.8,0) -- (-0.4,0);
                \draw[->] (0.4,0) -- (0.8,0);
                \node at (2,0) {$\Delta_{hgh^{-1},hkh^{-1}}$};
                \begin{scope}[shift = {(2.7,0)}]
                    \draw[->] (30:0.4) -- (30:0.8);
                    \draw[->] (-30:0.4) -- (-30:0.8);
                \end{scope}
            \end{scope}
        \end{tikzpicture}\vspace{0.5cm}

        \begin{tikzpicture}
            \node at (0,0) {$\phi_1$};
            \draw[->] (-0.8,0) -- (-0.4,0);
            \draw[->] (0.4,0) -- (0.8,0);
            \node at (1.1,0) {$\epsilon$};
            \node at (1.9,0) {$=$};
            \node at (3.4,0) {$\epsilon$};
            \draw[->] (2.7,0) -- (3.1,0);
        \end{tikzpicture}
    \end{center}

    and satisfies $\phi_{gh} = \phi_g\phi_h$ for all $g,h\in G$.
\end{dff}

It is easily verified that $\phi_h$ also satisfies 
    \begin{center}
        \begin{tikzpicture}
            \node at (0,0) {$\phi_h$};
            \draw[->] (-0.8,0) -- (-0.4,0);
            \draw[->] (0.4,0) -- (0.8,0);
            \node at (-1.1,0) {$S_g$};
            \draw[->] (-1.9,0) -- (-1.5,0);
            \node at (1.5,0) {$=$};
            \draw[->] (2.3,0) -- (2.7,0);
            \node at (3.1,0) {$\phi_h$};
            \draw[->] (3.4,0) -- (3.8,0);
            \node at (4.5,0) {$S_{hgh^{-1}}$};
            \draw[->] (5.2,0) -- (5.6,0);
        \end{tikzpicture}
    \end{center}
    for all $g,h\in G$.

    When $G$ is abelian, any Hopf $G$-coalgebra $H$ is crossed with $\phi_h\vert_{H_g} = id_{H_g}$.

\subsection{\texorpdfstring{Hopf $G$-Triplets}{Hopf G-Triplets}}

\begin{dff}
    Let $G$ be a group with identity $1$. A \emph{Hopf pair of Hopf $G$-coalgebras} (abbreviated to Hopf pair) is a triple $(H,H',\langle\_\rangle)$ where $H$ and $H'$ are Hopf $G$-coalgebras and $\langle\_\rangle$ is a bilinear form $H_1\otimes H'_1\to \Bbbk$ on the identity sectors such that the induced map $H_1\to (H'_1)^{*,cop}$ is a Hopf algebra morphism. We say that a Hopf pair is of \emph{finite type} if each $H$ and $H'$ is of finite type. Additionally, we say it is \emph{involutory} if each $H$ and $H'$ is involutory.
\end{dff}

\begin{dff}
    Let $G$ be a group with identity $1$. A \emph{Hopf $G$-doublet} is a triple $(H,H',\langle \_\rangle)$ where $H$ is a Hopf $G$-algebra, $H'$ is a Hopf $G$-coalgebra, and $\langle \_\rangle $ is a bilinear form $H\otimes H'\to \Bbbk$ such that the induced map $H\to (H')^{*,cop}$ is a Hopf $G$-algebra morphism. We say a Hopf $G$-doublet is of \emph{finite type} if each $H$ and $H'$ is of finite type. Additionally, we say it is \emph{involutory} if each $H$ and $H'$ is involutory.
\end{dff}

\begin{rmk}
    We remark that if $G$ is the trivial group, the notion of a Hopf pair and a Hopf $G$-doublet agree and match the definition of a Hopf doublet in \cite[Definition 2.10]{CCC22}.
\end{rmk}

From here on we assume that all Hopf $G$-coalgebras and algebras are involutory and of finite type unless otherwise stated.

\begin{lem}
    Let $(H,H',\langle\_\rangle)$ be a Hopf pair or a Hopf $G$-doublet. Then the following are as well:
    \begin{itemize}
        \item $(H^{op},(H')^{cop},\langle\_\rangle')$ where $\langle\_\rangle'_g = \langle\_\rangle_{g^{-1}}$
        \item $(H^{cop},(H')^{op},\langle\_\rangle)$
    \end{itemize}
    If $(H,H',\langle\_\rangle)$ is a Hopf pair, then the following are as well:
    \begin{itemize}
        \item $(H^{op,cop},H',\langle\_\rangle)$
        \item $(H,(H')^{op,cop},\langle\_\rangle)$
    \end{itemize}
\end{lem}

\begin{proof}
    Consider $(H^{op}, (H')^{cop}, \langle \_\rangle ')$, for example. If $(H,H',\langle\_\rangle)$ is a Hopf pair, then $\langle\_\rangle' = \langle\_\rangle$ and the result follows from the fact that for any finite-dimensional Hopf algebra $A$, $(A^{cop})^* = A^{*,op}$.

    If $(H,H',\langle\_\rangle)$ is a Hopf $G$-doublet, then the following equalities hold:

    \begin{center}
        \begin{tikzpicture}
            \node at (-1,0) (M) {$M_{h^{-1},g^{-1}}$};
            \node at (1,0) (B) {$\langle\_\rangle_{h^{-1}g^{-1}}$};
            \draw[<-] (M) .. controls +(-1.5,-0.25) and +(0.5,0) .. ++(-2,0.5);
            \draw[<-] (M) .. controls +(-1.5 ,0.25) and +(0.5,0) .. ++(-2,-0.5);
            \draw[->] (M) -- (B);
            \draw[<-] (B) -- ++(1.25,0);
            \node at (3,0) {$=$};
            \begin{scope}[shift = {(5,0)}]
                \node at (0,0.5) (B1) {$\langle\_\rangle_{g^{-1}}$};
                \node at (0,-0.5) (B2) {$\langle\_\rangle_{h^{-1}}$};
                \node at (2,0) (D) {$\Delta_{h^{-1},g^{-1}}$};
                \draw[<-] (B1) -- ++(-1,0);
                \draw[<-] (B2) -- ++(-1,0);
                \draw[<-] (B1) -- (D);
                \draw[<-] (B2) -- (D);
                \draw[<-] (D) -- ++(1.5,0);
            \end{scope}
            \begin{scope}[shift = {(0,-2)}]
                \node at (-1,0) (M) {$M^{op}_{g,h}$};
            \node at (1,0) (B) {$\langle\_\rangle_{h^{-1}g^{-1}}$};
            \draw[<-] (M) -- ++(-1,0.5);
            \draw[<-] (M) -- ++(-1,-0.5);
            \draw[->] (M) -- (B);
            \draw[<-] (B) -- ++(1.25,0);
            \node at (3,0) {$=$};
            \begin{scope}[shift = {(5,0)}]
                \node at (0,0.5) (B1) {$\langle\_\rangle_{g^{-1}}$};
                \node at (0,-0.5) (B2) {$\langle\_\rangle_{h^{-1}}$};
                \node at (2,0) (D) {$\Delta^{cop}_{g,h}$};
                \draw[<-] (B1) -- ++(-1,0);
                \draw[<-] (B2) -- ++(-1,0);
                \draw[<-] (B1) .. controls +(0.75,-0.25) and +(-0.75,-0.25) .. (D);
                \draw[<-] (B2) .. controls +(0.75,0.25) and +(-0.75,0.25) .. (D);
                \draw[<-] (D) -- ++(1,0);
            \end{scope}
            \end{scope}
        \end{tikzpicture}
    \end{center}

    This is part of the conditions for the induced map $H^{op}\to ((H')^{cop})^{*,cop}$ from $\langle\_\rangle'$ to be a Hopf $G$-algebra morphism. One can easily check the other conditions are implied.

    The other Hopf $G$-doublets and Hopf pairs are straightforward to check as well.
\end{proof}

For all bilinear forms $\langle\_\rangle$, we will often use the symbol $\bullet$ instead for brevity.

For a Hopf pair $(H^\alpha,H^\beta,\langle\_\rangle)$ and any $g,h\in G$, we will use the following notation:
\begin{center}
    \begin{tikzpicture}
        \node at (0,0) (T) {$T^{\alpha\beta}_{g,h}$};
        \draw[->] (T) -- ++(1,0.25);
        \draw[->] (T) -- ++(1,-0.25);
        \draw[<-] (T) -- ++(-1,0.25);
        \draw[<-] (T) -- ++(-1,-0.25);
        \node at (1.5,0) {$=$};
        \begin{scope}[shift = {(3,0)}]
            \node at (0,1) (DU) {$\Delta^\beta_{g,1}$};
            \node at (0,-1) (DL) {$\Delta^\alpha_{1,h}$};
            \node at (0,0) (B) {$\bullet$};
            \draw[<-] (DU) -- ++(-1,0);
            \draw[<-] (DL) -- ++(-1,0);
            \draw[->] (DU) -- ++(1,0);
            \draw[->] (DL) -- ++(1,0);
            \draw[->] (DU) -- (B);
            \draw[->] (DL) -- (B);
        \end{scope}
        \begin{scope}[shift = {(7,0)}]
        \node at (0,0) (T) {$(T^{\alpha\beta}_{g,h})^{-1}$};
        \draw[->] (T) -- ++(1,0.25);
        \draw[->] (T) -- ++(1,-0.25);
        \draw[<-] (T) -- ++(-1,0.25);
        \draw[<-] (T) -- ++(-1,-0.25);
        \node at (1.5,0) {$=$};
        \begin{scope}[shift = {(3,0)}]
            \node at (0,1) (DU) {$\Delta^\beta_{g,1}$};
            \node at (0,-1) (DL) {$\Delta^\alpha_{1,h}$};
            \node at (0,0.3) (B) {$\bullet$};
            \node at (0,-0.3) (S) {$S^\alpha_1$};
            \draw[<-] (DU) -- ++(-1,0);
            \draw[<-] (DL) -- ++(-1,0);
            \draw[->] (DU) -- ++(1,0);
            \draw[->] (DL) -- ++(1,0);
            \draw[->] (DU) -- (B);
            \draw[->] (DL) -- (S);
            \draw[->] (S) -- (B);
        \end{scope}
        \end{scope}
    \end{tikzpicture}\vspace{0.5cm}

    \begin{tikzpicture}
        \node at (0,0) (U) {$U^{\alpha\beta}_{g,h}$};
        \draw[->] (U) -- ++(1,0.25);
        \draw[->] (U) -- ++(1,-0.25);
        \draw[<-] (U) -- ++(-1,0.25);
        \draw[<-] (U) -- ++(-1,-0.25);
        \node at (1.5,0) {$=$};
        \begin{scope}[shift = {(3.75,0)}]
            \node at (0,0) (TI) {$(T^{\alpha\beta}_{g^{-1},h})^{-1}$};
            \node at (2,0) (T) {$T^{\alpha\beta}_{g,h^{-1}}$};
            \node at (-0.5,1) (SUL) {$S^\beta_g$};
            \node at (1,1) (SUR) {$S^\beta_{g^{-1}}$};
            \node at (1,-1) (SBL) {$S^\alpha_h$};
            \node at (2.5,-1) (SBR) {$S^\alpha_{h^{-1}}$};
            \draw[<-] (SUL) -- ++(-1,0);
            \draw[<-] (TI) -- ++(-1.5,-1);
            \draw[->] (SUL) -- (TI);
            \draw[->] (TI) -- (SUR);
            \draw[->] (SUR) -- (T);
            \draw[->] (TI) -- (SBL);
            \draw[->] (SBL) -- (T);
            \draw[->] (T) -- (SBR);
            \draw[->] (T) -- ++(1.5,1);
            \draw[->] (SBR) -- ++(1,0);
        \end{scope}
    \end{tikzpicture}\vspace{0.5cm}

    \begin{tikzpicture}
        \node at (0,0) (V) {$V^{\alpha\beta}_{g,h}$};
        \draw[->] (V) -- ++(1,0.25);
        \draw[->] (V) -- ++(1,-0.25);
        \draw[<-] (V) -- ++(-1,0.25);
        \draw[<-] (V) -- ++(-1,-0.25);
        \node at (1.5,0) {$=$};
        \begin{scope}[shift = {(4,0)}]
            \node at (0.5,0) (T) {$T^{\alpha\beta}_{g^{-1},h^{-1}}$};
            \node at (-1,0.5) (SUL) {$S^\beta_g$};
            \node at (-1,-0.5) (SBL) {$S^\alpha_h$};
            \node at (2,0.5) (SUR) {$S^\beta_{g^{-1}}$};
            \draw[<-] (SUL) -- ++(-1,0);
            \draw[<-] (SBL) -- ++(-1,0);
            \draw[->] (SUL) -- (T);
            \draw[->] (SBL) -- (T);
            \draw[->] (T) -- (SUR);
            \draw[->] (SUR) -- ++(1,0);
            \draw[->] (T) -- ++(2.5,-0.5);
        \end{scope}
    \end{tikzpicture}
\end{center}

\begin{lem}\label{TVU_lem}
    Let $(H^\alpha, H^\beta, \langle\_\rangle)$ be an involutory Hopf pair. Then for any $g,h\in G$ the tensors above satisfy the following relations
    \begin{center}
        \begin{tikzpicture}
            \node at (0,0) (V) {$V^{\alpha\beta}_{g,h}$};
            \node at (1.5,0) (U) {$U^{\alpha\beta}_{g,h^{-1}}$};
            \draw[<-] (V) -- ++(-1,0.25);
            \draw[<-] (V) -- ++(-1,-0.25);
            \draw[->] (V) .. controls +(0.5,0.5) and +(-0.5,0.5).. (U);
            \draw[->] (V) .. controls +(0.5,-0.5) and +(-0.5,-0.5) .. (U);
            \draw[->] (U) -- ++(1,0.25);
            \draw[->] (U) -- ++(1,-0.25);
            \node at (3.25,0) {$=$};
            \begin{scope}[shift = {(5,0)}]
                \node at (0,0) (T) {$T^{\alpha\beta}_{g,h}$};
                \node at (1,-0.5) (S) {$S^\alpha_h$};
                \draw[<-] (T) -- ++(-1,0.25);
                \draw[<-] (T) -- ++(-1,-0.25);
                \draw[->] (T) -- (S);
                \draw[->] (T) -- ++(2,0.25);
                \draw[->] (S) -- ++(1,0);
            \end{scope}
        \end{tikzpicture}\vspace{0.5cm}

        \begin{tikzpicture}
            \node at (-0.25,0) (T1) {$T^{\alpha\beta}_{g,h}$};
            \node at (2,0) (TI1) {$(T^{\alpha\beta}_{g^{-1},h^{-1}})^{-1}$};
            \node at (0.75,1) (SUL) {$S^\beta_g$};
            \node at (0.75,-1) (SBL) {$S^\alpha_h$};
            \node at (3.5,1) (SUR) {$S^\beta_{g^{-1}}$};
            \node at (3.5,-1) (SBR) {$S^\alpha_{h^{-1}}$};
            \draw[<-] (T1) -- ++(-0.75,0.5);
            \draw[<-] (T1) -- ++(-0.75,-0.5);
            \draw[->] (T1) -- (SUL);
            \draw[->] (T1) -- (SBL);
            \draw[->] (SUL) -- (TI1);
            \draw[->] (SBL) -- (TI1);
            \draw[->] (TI1) -- (SUR);
            \draw[->] (TI1) -- (SBR);
            \draw[->] (SUR) -- ++(0.75,0);
            \draw[->] (SBR) -- ++(0.75,0);
            \node at (5,0) {$=$};
            \begin{scope}[shift = {(8,0)}]
            \node at (-0.25,0) (TI1) {$(T^{\alpha\beta}_{g^{-1},h^{-1}})^{-1}$};
            \node at (2,0) (T1) {$T^{\alpha\beta}_{g,h}$};
            \node at (-1.5,1) (SUL) {$S^\beta_g$};
            \node at (-1.5,-1) (SBL) {$S^\alpha_h$};
            \node at (1,1) (SUR) {$S^\beta_{g^{-1}}$};
            \node at (1,-1) (SBR) {$S^\alpha_{h^{-1}}$};
            \draw[->] (T1) -- ++(0.75,0.5);
            \draw[->] (T1) -- ++(0.75,-0.5);
            \draw[<-] (T1) -- (SUR);
            \draw[<-] (T1) -- (SBR);
            \draw[->] (SUL) -- (TI1);
            \draw[->] (SBL) -- (TI1);
            \draw[->] (TI1) -- (SUR);
            \draw[->] (TI1) -- (SBR);
            \draw[<-] (SUL) -- ++(-0.6,0);
            \draw[<-] (SBL) -- ++(-0.6,0);
            \end{scope}
        \end{tikzpicture}
    \end{center}
\end{lem}

\begin{dff}
    Given a Hopf pair $(H,H',\langle\_\rangle)$, define the \emph{Drinfeld Double} $D(H,H')$ to be the Hopf $G$-coalgebra with the following structure: As a vector space, it is graded with grading $D(H,H')_g = H_g\otimes H'_g$, and has structure tensors
    \begin{center}
        \begin{tikzpicture}
            \node at (0,0) {$M^D_g$};
            \foreach \angle in {155,-135} {
                \draw[->] (\angle:0.9) -- (\angle:0.4);
            }
            \foreach \angle in {135,-155} {
                \draw[->] (\angle:0.9) -- (\angle:0.4);
            }
            \draw[->] (10:0.4) -- (10:0.9);
            \draw[->] (-10:0.4) -- (-10:0.9);
            \node at (1.75,0) {$=$};
            \begin{scope}[shift = {(4,0)}]
                \node at (0,0) (U) {$U_{g,g}$};
                \node at (1,1) (MU) {$M_g$};
                \node at (1,-1) (ML) {$M'_g$};
                \draw[<-] (U) -- ++(-1,0.25);
                \draw[<-] (U) -- ++(-1,-0.25);
                \draw[<-] (MU) -- ++(-2,0);
                \draw[<-] (ML) -- ++(-2,0);
                \draw[->] (U) .. controls +(0.5,-0.5) and +(0,-1) .. (MU);
                \draw[->] (U) .. controls +(0.5,0.5) and +(0,1) .. (ML);
                \draw[->] (MU) -- ++(1,0);
                \draw[->] (ML) -- ++(1,0);
            \end{scope}
        \end{tikzpicture}

        \begin{tikzpicture}
            \node at (0,0) {$\Delta^D_{g,h}$};
            \foreach \angle in {45,25,-45,-25} {
                \draw[<-] (\angle:1) -- (\angle:0.5);
            }
            \draw[<-] (10:-0.5) -- (10:-1);
            \draw[<-] (-10:-0.5) -- (-10:-1);
            \node at (2,0) {$=$};
            \begin{scope}[shift = {(4,0)}]
                \node at (0,0.5) (DT) {$\Delta_{g,h}$};
                \node at (0,-0.5) (DB) {$\Delta'_{g,h}$};
                \draw[->] (-1,0.5) -- (DT);
                \draw[->] (-1,-0.5) -- (DB);
                \draw[->] (DT) -- ++(1.5,1);
                \draw[->] (DT) -- ++(1.5,-1);
                \draw[->] (DB) -- ++(1.5,1);
                \draw[->] (DB) -- ++(1.5,-1);
            \end{scope}
        \end{tikzpicture}

        \begin{tikzpicture}
            \node at (0,0) {$S^D_{g}$};
            \foreach \angle in {170,-170} {
                \draw[->] (\angle:1) -- (\angle:0.5);
            }
            \draw[->] (10:0.5) -- (10:1);
            \draw[->] (-10:0.5) -- (-10:1);
            \node at (1.75,0) {$=$};
            \begin{scope}[shift = {(3.5,0)}]
                \node at (0,0.5) (SU) {$S_g$};
                \node at (0,-0.5) (SL) {$S'_g$};
                \node at (1.5,0) (U) {$U_{g^{-1},g^{-1}}$};
                \draw[<-] (SU) -- ++(-1,0);
                \draw[<-] (SL) -- ++(-1,0);
                \draw[->] (SU) -- (U);
                \draw[->] (SL) -- (U);
                \draw[->] (U) -- ++(1.25,0.25);
                \draw[->] (U) -- ++(1.25,-0.25);
            \end{scope}
        \end{tikzpicture}
    \end{center}
\end{dff}

\begin{dff}\label{HopfGTriplet_def}
    Let $G$ be a group with identity 1. A \emph{Hopf $G$-triplet} is a tuple \[(H^\alpha,H^\beta,H^\kappa,\langle \_\rangle ^{\kappa,\beta},\langle \_\rangle^{\alpha,\beta}, \langle\_\rangle^{\alpha,\kappa})\] where $H^\alpha$ is a Hopf $G$-algebra, $H^\beta$ and $H^\kappa$ are Hopf $G$-coalgebras such that:
    \begin{enumerate}[(a)]
        \item $(H^\kappa,H^\beta,\langle \_\rangle^{\kappa,\beta})$ is a Hopf pair
        \item $(H^\alpha, H^\beta,\langle\_\rangle^{\alpha,\beta})$ and $(H^\alpha, H^\kappa, \langle\_\rangle^{\alpha,\kappa})$ are Hopf $G$-doublets
        \item The family of maps $D(H^{\beta}, H^{\kappa})_g\to H^{\alpha,*}_g$ defined by:
        \begin{center}
            \begin{tikzpicture}
                \node at (0,0.5) (A) {$\bullet$};
                \node at (0,-0.5) (B) {$\bullet$};
                \draw[->] (-0.6,0.5) -- (A);
                \draw[->] (-0.6,-0.5) -- (B);
                \node at (1.2,0) (D) {$\Delta^\alpha_{g}$};
                \draw[<-] (A) .. controls +(0.5,-0.5) and +(-0.5,-0.5) .. (D);
                \draw[<-] (B) .. controls +(0.5,0.5) and +(-0.5,0.5) .. (D);
                \draw[<-] (D) -- (2.2,0);
            \end{tikzpicture}
        \end{center}
        form a Hopf $G$-coalgebra morphism
    \end{enumerate}
    We say that a Hopf $G$-triplet is of \emph{finite type} if each $H^\alpha$, $H^\beta$, and $H^\kappa$ is of finite type. Additionally, we say it is \emph{involutory} if each $H^\alpha$, $H^\beta$, and $H^\kappa$ is involutory.

    A \emph{morphism} of Hopf $G$-triplets is a triple $(f^\alpha, f^\beta, f^\kappa)$ consisting of a morphism of Hopf $G$-algebras $f^\alpha$, and two morphisms of Hopf $G$-coalgebras $f^\beta$ and $f^\kappa$ which intertwine the pairwise bilinear forms.
\end{dff}

\begin{rmk}
 When $G = 1$ is trivial, the definition given here is slightly different to the definition of Hopf triplet given in \cite[Definition 2.16]{CCC22}. In particular, their convention is to use a cyclic ordering of the indices $\alpha, \beta, \kappa$, and their Drinfeld double is defined using $H^{\beta,op}$ and $H^{\kappa,cop}$. The reason for our convention here is due to the inversion of the index when utilizing $H^{cop}$ for a Hopf $G$-coalgebra $H$ or, equivalently, $H^{op}$ for a Hopf $G$-algebra $H$. We note that this change does not lose any generality.
\end{rmk}

There is a fundamental lemma of Hopf $G$-triplets a la \cite[Lemma 2.19]{CCC22}, which we state here:

\begin{lem}[Fundamental Lemma of Hopf $G$-Triplets]\label{fundlemoftriplets_lem}
    Let $H^\alpha$ be a Hopf $G$-algebra, $H^\beta$ and $H^\kappa$ Hopf $G$-coalgebras, and $\langle\_\rangle^{\kappa,\beta}$, $\langle\_\rangle^{\alpha,\beta}$, and $\langle\_\rangle^{\alpha,\kappa}$ be bilinear forms all of which satisfy all but part (c) of definition \ref{HopfGTriplet_def}, then the following are equivalent:
    \begin{enumerate}[(a)]
        \item $(H^\alpha,H^\beta,H^\kappa,\langle\_\rangle^{\kappa,\beta},\langle\_\rangle^{\alpha,\beta},\langle\_\rangle^{\alpha,\kappa})$ is a Hopf $G$-triplet
        \item The following equation holds:
        
            \begin{center}
                \begin{tikzpicture}
                    \node at (-1.25,1.75) (DUL) {$\Delta^\kappa_{g,1}$};
                    \node at (1.25,1.75) (BUR) {$\langle\_\rangle^{\alpha,\kappa}_g$};
                    \node at (-1.25,0) (BML) {$\langle\_\rangle^{\kappa,\beta}$};
                    \node at (-1.25,-1.5) (DBL) {$\Delta^\beta_{1,g^{-1}}$};
                    \node at (1.25,-1.5) (BBR) {$\langle\_\rangle^{\alpha,\beta}_{g}$};
                    \node at (1.25,0) (DMR) {$\Delta^\alpha_{g}$};
                    \node at (0,-1.5) (SB) {$S^\beta_{g^{-1}}$};
                    \draw[->] (DUL) -- (BUR);
                    \draw[->] (DUL) -- (BML);
                    \draw[->] (DMR) -- (BUR);
                    \draw[->] (DMR) -- (BBR);
                    \draw[->] (DBL) -- (BML);
                    \draw[->] (DBL) -- (SB);
                    \draw[->] (SB) -- (BBR);
                    \draw[<-] (DUL) -- ++(-1,0);
                    \draw[<-] (DBL) -- ++(-1,0);
                    \draw[<-] (DMR) -- ++(1,0);
                    \begin{scope}[shift = {(7.5,0)}]
                    \node at (-4.5,0) {$=$};
                    \node at (-1.25,1.75) (DUL) {$\Delta^\kappa_{g^{-1},1}$};
                    \node at (1.25,1.75) (BUR) {$\langle\_\rangle^{\alpha,\kappa}_{g^{-1}}$};
                    \node at (-1.25,0) (BML) {$\langle\_\rangle^{\kappa,\beta}$};
                    \node at (-1.25,-1.5) (DBL) {$\Delta^\beta_{1,g}$};
                    \node at (1.25,-1.5) (BBR) {$\langle\_\rangle^{\alpha,\beta}_{g^{-1}}$};
                    \node at (1.25,0) (DMR) {$\Delta^\alpha_{g^{-1}}$};
                    \node at (0,-1.5) (SB) {$S^\beta_{g}$};
                    \node at (-2.75,1.75) (SUL) {$S^\kappa_{g}$};
                    \node at (-2.75,-1.5) (SBL) {$S^\beta_{g^{-1}}$};
                    \node at (2.75,0) (SMR) {$S^\alpha_{g}$};
                    \draw[->] (DUL) -- (BUR);
                    \draw[->] (DUL) -- (BML);
                    \draw[->] (DMR) -- (BUR);
                    \draw[->] (DMR) -- (BBR);
                    \draw[->] (DBL) -- (BML);
                    \draw[->] (DBL) -- (SB);
                    \draw[->] (SB) -- (BBR);
                    \draw[<-] (DUL) -- (SUL);
                    \draw[<-] (DBL) -- (SBL);
                    \draw[<-] (DMR) -- (SMR);
                    \draw[<-] (SUL) -- ++(-1,0);
                    \draw[<-] (SBL) -- ++(-1,0);
                    \draw[<-] (SMR) -- ++(1,0);
                    \end{scope}
            \end{tikzpicture}
        \end{center}
        \item The following equation holds:

        \begin{center}
            \begin{tikzpicture}
                    \node at (-1.25,1.75) (DUL) {$\Delta^\kappa_{g^{-1},1}$};
                    \node at (1.25,1.75) (BUR) {$\langle\_\rangle^{\alpha,\kappa}_g$};
                    \node at (-1.25,0.5) (BML) {$\langle\_\rangle^{\kappa,\beta}$};
                    \node at (-1.25,-1.5) (DBL) {$\Delta^\beta_{1,g}$};
                    \node at (1.25,-1.5) (BBR) {$\langle\_\rangle^{\alpha,\beta}_{g}$};
                    \node at (1.25,0) (DMR) {$\Delta^\alpha_{g}$};
                    \node at (0,1.75) (SU) {$S^\kappa_{g^{-1}}$};
                    \node at (-1.25,-0.5) (SL) {$S^\beta_1$};
                    \draw[->] (DUL) -- (SU);
                    \draw[->] (SU) -- (BUR);
                    \draw[->] (DUL) -- (BML);
                    \draw[->] (DMR) -- (BUR);
                    \draw[->] (DMR) -- (BBR);
                    \draw[->] (DBL) -- (SL);
                    \draw[->] (SL) -- (BML);
                    \draw[->] (DBL) -- (BBR);
                    \draw[<-] (DUL) -- ++(-1,0);
                    \draw[<-] (DBL) -- ++(-1,0);
                    \draw[<-] (DMR) -- ++(1,0);
                    \begin{scope}[shift = {(7.5,0)}]
                    \node at (-4.5,0) {$=$};
                    \node at (-1.25,1.75) (DUL) {$\Delta^\kappa_{g,1}$};
                    \node at (1.25,1.75) (BUR) {$\langle\_\rangle^{\alpha,\kappa}_{g^{-1}}$};
                    \node at (-1.25,0.5) (BML) {$\langle\_\rangle^{\kappa,\beta}$};
                    \node at (-1.25,-1.5) (DBL) {$\Delta^\beta_{1,g^{-1}}$};
                    \node at (1.25,-1.5) (BBR) {$\langle\_\rangle^{\alpha,\beta}_{g^{-1}}$};
                    \node at (1.25,0) (DMR) {$\Delta^\alpha_{g^{-1}}$};
                    \node at (0,1.75) (SU) {$S^\kappa_{g}$};
                    \node at (-1.25,-0.5) (SL) {$S^\beta_{1}$};
                    \node at (-2.75,1.75) (SUL) {$S^\kappa_{g^{-1}}$};
                    \node at (-2.75,-1.5) (SBL) {$S^\beta_{g}$};
                    \node at (2.75,0) (SMR) {$S^\alpha_{g}$};
                    \draw[->] (DUL) -- (SU);
                    \draw[->] (SU) -- (BUR);
                    \draw[->] (DUL) -- (BML);
                    \draw[->] (DMR) -- (BUR);
                    \draw[->] (DMR) -- (BBR);
                    \draw[->] (DBL) -- (SL);
                    \draw[->] (SL) -- (BML);
                    \draw[->] (DBL) -- (BBR);
                    \draw[<-] (DUL) -- (SUL);
                    \draw[<-] (DBL) -- (SBL);
                    \draw[<-] (DMR) -- (SMR);
                    \draw[<-] (SUL) -- ++(-1,0);
                    \draw[<-] (SBL) -- ++(-1,0);
                    \draw[<-] (SMR) -- ++(1,0);
                    \end{scope}
            \end{tikzpicture}
        \end{center}
    \end{enumerate}
\end{lem}

\begin{proof}
    Let $\Phi_g$ denote the map $D(H^{\beta},H^{\kappa})\to (H^\alpha_g)^*$. First we show (a)$\Leftrightarrow$(c). Note that the equality in (c) is equivalent to the following sequence of equalities:

        \begin{center}
            \begin{tikzpicture}
                \node at (0,0) (TI) {$(T^{\beta,\kappa}_{g^{-1},g})^{-1}$};
                \node at (1,1) (S) {$S^\kappa_{g^{-1}}$};
                \node at (2.5,0) (Phi) {$\Phi_g$};
                \draw[->] (TI) -- (S);
                \draw[->] (TI) .. controls +(0.5,-1) and +(-0.5,1) .. (Phi);
                \draw[->] (S) .. controls +(0.5,-1) and +(-0.5,-1) .. (Phi);
                \draw[<-] (TI) -- ++(-1,0.5);
                \draw[<-] (TI) -- ++(-1,-0.5);
                \draw[<-] (Phi) -- ++(0.75,0);
                \node at (4,0) {$=$};
                \begin{scope}[shift = {(6.5,0)}]
                    \node at (0,0) (TI) {$(T^{\beta,\kappa}_{g,g^{-1}})^{-1}$};
                \node at (1,1) (SU) {$S^\kappa_{g}$};
                \node at (-1,1) (SUL) {$S^\kappa_{g^{-1}}$};
                \node at (-1,-1) (SBL) {$S^\beta_g$};
                \node at (3.5,0) (SR) {$S^\alpha_g$};
                \node at (2.5,0) (Phi) {$\Phi_{g^{-1}}$};
                \draw[->] (TI) -- (SU);
                \draw[->] (TI) .. controls +(0.5,-1) and +(-0.5,1) .. (Phi);
                \draw[->] (SU) .. controls +(0.5,-1) and +(-0.5,-1) .. (Phi);
                \draw[<-] (TI) -- (SUL);
                \draw[<-] (TI) -- (SBL);
                \draw[<-] (Phi) -- (SR);
                \draw[<-] (SR) -- ++(0.75,0);
                \draw[<-] (SUL) -- ++(-0.75,0);
                \draw[<-] (SBL) -- ++(-0.75,0);
                \end{scope}
            \end{tikzpicture}
        \end{center}

        \begin{center}
            \begin{tikzpicture}
                \node at (0,0) (TI) {$(T^{\beta,\kappa}_{g^{-1},g})^{-1}$};
                \node at (1,1) (S) {$S^\kappa_{g^{-1}}$};
                \node at (-1,1) (SUL) {$S^\kappa_g$};
                \node at (-1,-1) (SBL) {$S^\beta_{g^{-1}}$};
                \node at (-2,0) (T) {$T^{\beta,\kappa}_{g,g^{-1}}$};
                \node at (-3,-1) (SBLL) {$S^\beta_g$};
                \node at (2.5,0) (Phi) {$\Phi_g$};
                \draw[->] (TI) -- (S);
                \draw[->] (TI) .. controls +(0.5,-1) and +(-0.5,1) .. (Phi);
                \draw[->] (S) .. controls +(0.5,-1) and +(-0.5,-1) .. (Phi);
                \draw[<-] (TI) -- (SUL);
                \draw[<-] (TI) -- (SBL);
                \draw[->] (T) -- (SUL);
                \draw[->] (T) -- (SBL);
                \draw[->] (SBLL) -- (T);
                \draw[<-] (SBLL) -- ++(-0.75,0);
                \draw[<-] (T) -- ++(-1.75,0.5);
                \draw[<-] (Phi) -- ++(0.75,0);
                \node at (4,0) {$=$};
                \begin{scope}[shift = {(4.5,0)}]
                    \node at (1,-1) (SL) {$S^\beta_g$};
                \node at (1,1) (SU) {$S^\kappa_{g}$};
                \node at (3.5,0) (SR) {$S^\alpha_g$};
                \node at (2.5,0) (Phi) {$\Phi_{g^{-1}}$};
                \draw[->] (SL) .. controls +(0.5,1) and +(-0.5,1) .. (Phi);
                \draw[->] (SU) .. controls +(0.5,-1) and +(-0.5,-1) .. (Phi);
                \draw[<-] (Phi) -- (SR);
                \draw[<-] (SR) -- ++(0.75,0);
                \draw[<-] (SL) -- ++(-0.75,0);
                \draw[<-] (SU) -- ++(-0.75,0);
                \end{scope}
            \end{tikzpicture}
        \end{center}

        \begin{center}
            \begin{tikzpicture}
                \node at (-2,0) (TI) {$(T^{\beta,\kappa}_{g^{-1},g})^{-1}$};
                \node at (1,-1) (S) {$S^\beta_{g^{-1}}$};
                \node at (-1,1) (SUL) {$S^\kappa_{g^{-1}}$};
                \node at (-1,-1) (SBL) {$S^\beta_{g}$};
                \node at (0,0) (T) {$T^{\beta,\kappa}_{g,g^{-1}}$};
                \node at (-3,1) (SULL) {$S^\kappa_g$};
                \node at (2.5,0) (Phi) {$\Phi_g$};
                \draw[->] (T) -- (S);
                \draw[->] (T) .. controls +(0.5,1) and +(-0.5,-1) .. (Phi);
                \draw[->] (S) .. controls +(0.5,1) and +(-0.5,1) .. (Phi);
                \draw[<-] (T) -- (SUL);
                \draw[<-] (T) -- (SBL);
                \draw[->] (TI) -- (SUL);
                \draw[->] (TI) -- (SBL);
                \draw[->] (SULL) -- (TI);
                \draw[<-] (SULL) -- ++(-0.75,0);
                \draw[<-] (TI) -- ++(-1.75,-0.5);
                \draw[<-] (Phi) -- ++(0.75,0);
                \node at (4,0) {$=$};
                \begin{scope}[shift = {(4.5,0)}]
                    \node at (1,-1) (SL) {$S^\beta_g$};
                \node at (1,1) (SU) {$S^\kappa_{g}$};
                \node at (3.5,0) (SR) {$S^\alpha_g$};
                \node at (2.5,0) (Phi) {$\Phi_{g^{-1}}$};
                \draw[->] (SL) .. controls +(0.5,1) and +(-0.5,1) .. (Phi);
                \draw[->] (SU) .. controls +(0.5,-1) and +(-0.5,-1) .. (Phi);
                \draw[<-] (Phi) -- (SR);
                \draw[<-] (SR) -- ++(0.75,0);
                \draw[<-] (SL) -- ++(-0.75,0);
                \draw[<-] (SU) -- ++(-0.75,0);
                \end{scope}
            \end{tikzpicture}
        \end{center}

        \begin{center}
            \begin{tikzpicture}
                \node at (0,0) (U) {$U^{\beta,\kappa}_{g,g}$};
                \node at (1.5,0) (Phi) {$\Phi_g$};
                \draw[->] (U) .. controls +(0.5,1) and +(-0.5,-1) .. (Phi);
                \draw[->] (U) .. controls +(0.5,-1) and +(-0.5,1) .. (Phi);
                \draw[<-] (U) -- ++(-.75,-0.5);
                \draw[<-] (U) -- ++(-.75,0.5);
                \draw[<-] (Phi) -- ++(0.75,0);
                \node at (3,0) {$=$};
                \begin{scope}[shift = {(3.5,0)}]
                    \node at (1,-1) (SL) {$S^\beta_g$};
                \node at (1,1) (SU) {$S^\kappa_{g}$};
                \node at (3.5,0) (SR) {$S^\alpha_g$};
                \node at (2.5,0) (Phi) {$\Phi_{g^{-1}}$};
                \draw[->] (SL) .. controls +(0.5,1) and +(-0.5,1) .. (Phi);
                \draw[->] (SU) .. controls +(0.5,-1) and +(-0.5,-1) .. (Phi);
                \draw[<-] (Phi) -- (SR);
                \draw[<-] (SR) -- ++(0.75,0);
                \draw[<-] (SL) -- ++(-0.75,0);
                \draw[<-] (SU) -- ++(-0.75,0);
                \end{scope}
            \end{tikzpicture}
        \end{center}

        which holds if and only if $\Phi_g$ is a Hopf $G$-coalgebra morphism.

        Now we show (b)$\Leftrightarrow$(c). The two equalities are, respectively

        \begin{center}
            \begin{tikzpicture}
                \node at (-1,0) (T) {$T^{\beta,\kappa}_{g,g^{-1}}$};
                \node at (0,-1) (S) {$S^\beta_{g^{-1}}$};
                \node at (1,0.5) (BU) {$\bullet$};
                \node at (1,-0.5) (BL) {$\bullet$};
                \node at (1.75,0) (D) {$\Delta^\alpha_g$};
                \draw[<-] (T) -- ++(-1,0.5);
                \draw[<-] (T) -- ++(-1,-0.5);
                \draw[->] (T) -- (BU);
                \draw[->] (T) -- (S);
                \draw[->] (S) -- (BL);
                \draw[->] (D) -- (BU);
                \draw[->] (D) -- (BL);
                \draw[<-] (D) -- ++(0.75,0);
            \begin{scope}[shift = {(6,0)}]
                \node at (-3,0) {$=$};
                \node at (-1,0) (T) {$T^{\beta,\kappa}_{g^{-1},g}$};
                \node at (0,-1) (S) {$S^\beta_{g}$};
                \node at (-2,1) (SU) {$S^\kappa_{g}$};
                \node at (-2,-1) (SL) {$S^\beta_{g^{-1}}$};
                \node at (2.75,0) (SR) {$S^\alpha_g$};
                \node at (0.75,0.5) (BU) {$\bullet$};
                \node at (0.75,-0.5) (BL) {$\bullet$};
                \node at (1.75,0) (D) {$\Delta^\alpha_{g^{-1}}$};
                \draw[<-] (T) -- (SU);
                \draw[<-] (T) -- (SL);
                \draw[->] (T) -- (BU);
                \draw[->] (T) -- (S);
                \draw[->] (S) -- (BL);
                \draw[->] (D) -- (BU);
                \draw[->] (D) -- (BL);
                \draw[<-] (D) -- (SR);
                \draw[<-] (SR) -- ++(0.75,0);
                \draw[<-] (SU) -- ++(-0.75,0);
                \draw[<-] (SL) -- ++(-0.75,0);
            \end{scope}
            \end{tikzpicture}
        
            \begin{tikzpicture}
                \node at (0,0) (T) {$\Big(T^{\beta,\kappa}_{g^{-1},g}\Big)^{-1}$};
                \node at (1.75,0.5) (BU) {$\bullet$};
                \node at (1.75,-0.5) (BL) {$\bullet$};
                \node at (2.5,0) (D) {$\Delta^\alpha_{g}$};
                \node at (1,1) (SUR) {$S^{\kappa}_{g}$};
                \draw[->] (T) -- (SUR);
                \draw[->] (SUR) -- (BU);
                \draw[->] (T) -- (BL);
                \draw[<-] (BU) -- (D);
                \draw[<-] (BL) -- (D);
                \draw[<-] (T) -- ++(-1.5,0.5);
                \draw[<-] (T) -- ++(-1.5,-0.5);
                \draw[<-] (D) -- ++(0.75,0);
                \begin{scope}[shift = {(5.5,0)}]
                \node at (-2,0) {$=$};
                \node at (0,0) (T) {$\Big(T^{\beta,\kappa}_{g,g^{-1}}\Big)^{-1}$};
                \node at (2,0.5) (BU) {$\bullet$};
                \node at (2,-0.5) (BL) {$\bullet$};
                \node at (3,0) (D) {$\Delta^\alpha_{g^{-1}}$};
                \node at (1,1) (SUR) {$S^{\kappa}_{g^{-1}}$};
                \node at (-1,1) (SUL) {$S^\kappa_{g}$};
                \node at (-1,-1) (SBL) {$S^\beta_{g^{-1}}$};
                \node at (4,0) (SR) {$S^\alpha_{g}$};
                \draw[->] (SUL) -- (T);
                \draw[->] (SBL) -- (T);
                \draw[->] (T) -- (SUR);
                \draw[->] (SUR) -- (BU);
                \draw[->] (T) -- (BL);
                \draw[<-] (BU) -- (D);
                \draw[<-] (BL) -- (D);
                \draw[->] (SR) -- (D);
                \draw[<-] (SUL) -- ++(-0.75,0);
                \draw[<-] (SBL) -- ++(-0.75,0);
                \draw[<-] (SR) -- ++(0.75,0);
                \end{scope}
            \end{tikzpicture}
        \end{center}

        Isolating the $\alpha$ parts on the RHS of both equalities yields:

        \begin{center}
            \begin{tikzpicture}
                \node at (-1,0) (T) {$T^{\beta,\kappa}_{g,g^{-1}}$};
                \node at (0,-1) (S) {$S^\beta_{g^{-1}}$};
                \node at (-2,1) (SU) {$S^\kappa_{g^{-1}}$};
                \node at (-2,-1) (SB) {$S^\beta_g$};
                \node at (-3,0) (TI) {$(T^{\beta,\kappa}_{g^{-1},g})^{-1}$};
                \node at (-4,-1) (SBL) {$S^\beta_{g^{-1}}$};
                \node at (1,0.5) (BU) {$\bullet$};
                \node at (1,-0.5) (BL) {$\bullet$};
                \node at (1.75,0) (D) {$\Delta^\alpha_g$};
                \draw[<-] (T) -- (SU);
                \draw[<-] (T) -- (SB);
                \draw[->] (T) -- (BU);
                \draw[->] (T) -- (S);
                \draw[->] (S) -- (BL);
                \draw[->] (D) -- (BU);
                \draw[->] (D) -- (BL);
                \draw[->] (SBL) -- (TI);
                \draw[->] (TI) -- (SU);
                \draw[->] (TI) -- (SB);
                \draw[<-] (TI) -- ++(-1.75,0.5);
                \draw[<-] (SBL) -- ++(-0.75,0);
                \draw[<-] (D) -- ++(0.75,0);
            \begin{scope}[shift = {(3.5,0)}]
                \node at (-0.5,0) {$=$};
                \node at (2.75,0) (SR) {$S^\alpha_g$};
                \node at (0.75,0.5) (BU) {$\bullet$};
                \node at (0.75,-0.5) (BL) {$\bullet$};
                \node at (1.75,0) (D) {$\Delta^\alpha_{g^{-1}}$};
                \draw[<-] (BU) -- ++(-0.75,0);
                \draw[<-] (BL) -- ++(-0.75,0);
                \draw[->] (D) -- (BU);
                \draw[->] (D) -- (BL);
                \draw[<-] (D) -- (SR);
                \draw[<-] (SR) -- ++(0.75,0);
            \end{scope}
            \end{tikzpicture}
        
            \begin{tikzpicture}
                \node at (-1,0) (T) {$(T^{\beta,\kappa}_{g^{-1},g})^{-1}$};
                \node at (0,1) (S) {$S^\kappa_{g^{-1}}$};
                \node at (-2,1) (SU) {$S^\kappa_{g}$};
                \node at (-2,-1) (SB) {$S^\beta_{g^{-1}}$};
                \node at (-3,0) (TI) {$T^{\beta,\kappa}_{g,g^{-1}}$};
                \node at (-4,1) (SUL) {$S^\kappa_{g^{-1}}$};
                \node at (1,0.5) (BU) {$\bullet$};
                \node at (1,-0.5) (BL) {$\bullet$};
                \node at (1.75,0) (D) {$\Delta^\alpha_g$};
                \draw[<-] (T) -- (SU);
                \draw[<-] (T) -- (SB);
                \draw[->] (T) -- (BL);
                \draw[->] (T) -- (S);
                \draw[->] (S) -- (BU);
                \draw[->] (D) -- (BU);
                \draw[->] (D) -- (BL);
                \draw[->] (SUL) -- (TI);
                \draw[->] (TI) -- (SU);
                \draw[->] (TI) -- (SB);
                \draw[<-] (TI) -- ++(-1.75,-0.5);
                \draw[<-] (SUL) -- ++(-0.75,0);
                \draw[<-] (D) -- ++(0.75,0);
            \begin{scope}[shift = {(3.5,0)}]
                \node at (-0.5,0) {$=$};
                \node at (2.75,0) (SR) {$S^\alpha_g$};
                \node at (0.75,0.5) (BU) {$\bullet$};
                \node at (0.75,-0.5) (BL) {$\bullet$};
                \node at (1.75,0) (D) {$\Delta^\alpha_{g^{-1}}$};
                \draw[<-] (BU) -- ++(-0.75,0);
                \draw[<-] (BL) -- ++(-0.75,0);
                \draw[->] (D) -- (BU);
                \draw[->] (D) -- (BL);
                \draw[<-] (D) -- (SR);
                \draw[<-] (SR) -- ++(0.75,0);
            \end{scope}
            \end{tikzpicture}
        \end{center}

        Comparing the LHS of these two equalities yields two tensors the equality of which follows from Lemma \ref{TVU_lem}.
\end{proof}

\begin{lem}\label{Mochida_Triplet_Lem}
    Let $H$ be any finite-type involutory quasi-triangular Hopf $G$-algebra with universal R-matrix $R\in H_1\otimes H_1$, defined as in \cite[Definition 2.26]{Moc24}, then \[(H^{cop}, H^*, H^*,\langle \_\rangle_R, \langle\_\rangle ,\langle\_\rangle)\] 
    \begin{center}
        \begin{tikzpicture}
            \node at (0,0) (B) {$\langle\_\rangle_R$};
            \foreach \angle in {150,-150} {
                \draw[<-] (\angle:0.8) -- (B);
            }
            \node at (1.5,0) {$=$};
            \begin{scope}[shift = {(3.2,0)}]
                \node at (0,0) (R) {$R$};
            \foreach \angle in {150,-150} {
                \draw[<-] (\angle:0.8) -- (R);
            }
            \node at (-0.45,-0.45) {\resizebox{4pt}{!}{1}};
            \end{scope}
        \end{tikzpicture}
    \end{center}
    where the little 1 indicates the first leg of the R-matrix and the pairing $\langle \_\rangle$ is the pairing which induces the identity.
\end{lem}

\begin{proof}
We just need to check that the map $D(H^*,H^*)_g\to (H^{cop})^*_g$ defined by
\begin{center}
    \begin{tikzpicture}
                \node at (0,0) (D) {$\Delta_g$};
                \draw[->] (D) -- ++(-1,0.5);
                \draw[->] (D) -- ++(-1,-0.5);
                \draw[<-] (D) -- ++(1,0);
    \end{tikzpicture}
\end{center}
is a Hopf $G$-coalgebra morphism. By virtue of Lemma \ref{fundlemoftriplets_lem} and since $R^{-1} = (S_1\otimes id_{H_1})(R) = (id_{H_1}\otimes S_1)(R)$, we verify the following equality:
    \begin{center}
                \begin{tikzpicture}
                    \node at (-1,1.5) (DUL) {$M^{op}_{1,g^{-1}}$};
                    \node at (-1,0) (BML) {$R$};
                    \node at (-1,-1.5) (DBL) {$M^{op}_{g,1}$};
                    \node at (1,0) (DMR) {$\Delta^{cop}_{g}$};
                    \node at (-1.2,0.5) {\tiny{1}};
                    \node at (1,-1.5) (S) {$S_g$};
                    \draw[<-] (DUL) -- (BML);
                    \draw[->] (DMR) -- (S);
                    \draw[->] (S) -- (DBL);
                    \draw[<-] (DBL) -- (BML);
                    \draw[<-] (DUL) -- (DMR);
                    \draw[->] (DUL) -- ++(-1,0);
                    \draw[->] (DBL) -- ++(-1,0);
                    \draw[<-] (DMR) -- ++(1,0);
                    \node at (2.75,0) {$=$};
                    \node at (2.75,-4) {$=$};
                    \begin{scope}[shift = {(6.25,0)}]
                    \node at (-1,1.5) (DUL) {$M^{op}_{1,g}$};
                    \node at (-1,0) (BML) {$R$};
                    \node at (-1,-1.5) (DBL) {$M^{op}_{g^{-1},1}$};
                    \node at (1,0) (DMR) {$\Delta^{cop}_{g^{-1}}$};
                    \node at (-2.25,1.5) (SUL) {$S_{g^{-1}}$};
                    \node at (-2.25,-1.5) (SBL) {$S_g$};
                    \node at (2.25,0) (SMR) {$S_g$};
                    \node at (-1.2,0.5) {\tiny{1}};
                    \node at (1,-1.5) (S) {$S_{g^{-1}}$};
                    \draw[<-] (DBL) -- (S);
                    \draw[<-] (S) -- (DMR);
                    \draw[<-] (DUL) -- (BML);
                    \draw[->] (DMR) -- (DUL);
                    \draw[<-] (DBL) -- (BML);
                    \draw[->] (DUL) -- (SUL);
                    \draw[->] (DBL) -- (SBL);
                    \draw[<-] (DMR) -- (SMR);
                    \draw[->] (SUL) -- ++(-0.75,0);
                    \draw[->] (SBL) -- ++(-0.75,0);
                    \draw[<-] (SMR) -- ++(0.75,0);
                    \end{scope}
                    \begin{scope}[shift = {(6.25,-4)}]
                    \node at (-1,1.5) (DUL) {$M_{1,g}$};
                    \node at (-1,0) (BML) {$R$};
                    \node at (-1,-1.5) (DBL) {$M_{g^{-1},1}$};
                    \node at (1,0) (DMR) {$\Delta_{g}$};
                    \node at (-1.2,0.5) {\tiny{1}};
                    \node at (1,-1.5) (S) {$S_g$};
                    \draw[<-] (DBL) -- (S);
                    \draw[<-] (S) -- (DMR);
                    \draw[<-] (DUL) -- (BML);
                    \draw[->] (DMR) -- (DUL);
                    \draw[<-] (DBL) -- (BML);
                    \draw[->] (DUL) -- ++(-1,0);
                    \draw[->] (DBL) -- ++(-1,0);
                    \draw[<-] (DMR) -- ++(1,0);
                    \end{scope}
                \end{tikzpicture}
    \end{center}

    The first and last tensors are equal because $R$ is a universal R-matrix for $H$.
\end{proof}

\section{Invariants of Colored Trisection Diagrams}
\label{ColoredTrisectionInvariants_Sec}
\subsection{Trisections of 4-Manifolds}

Let $Y^k_- \cup Y^k_+ = \mathop{\hbox{\large$\#$}}_{i=1}^k S^1\times S^2$ be the standard Heegaard splitting for the $k$-fold connected sum of $S^1\times S^2$, i.e. $Y^k_\pm$ are standard handlebodies of genus $k$. When $k=0$, set $\mathop{\hbox{\large$\#$}}_{i=1}^k S^1\times S^2 = S^3$. Denote by $A\natural B$ the boundary connected sum of $A$ and $B$.

\begin{dff}
    A \emph{trisection} of a closed connected orientable 3-manifold $X$ is a sextuple $(X_1,\phi_1,X_2,\phi_2,X_3,\phi_3)$ where $X_i$ are 4-manifolds with boundary, and $\phi_i$ are diffeomorphisms satisfying:
    \begin{enumerate}
        \item There exists a $k\in \mbb Z_{\geq 0}$ such that each $\phi_i$ specifies a diffeomorphism $X_i\cong \mathop{\hbox{\large$\natural$}}_{i=1}^k S^1\times D^3 $
        \item With indices modulo 3, we have diffeomorphisms $\phi_i(X_i\cap X_{i+1}) \cong Y^k_-$ and $\phi_i(X_i\cap X_{i-1}) \cong Y^k_+$.
    \end{enumerate}
\end{dff}

\begin{dff}
    An \emph{(oriented) trisection diagram} $T$ is a triple $(\Sigma,\alpha,\beta,\kappa)$ consisting of the following data:
    \begin{enumerate}[(a)]
        \item A closed, oriented 2-manifold $\Sigma$ of genus $g$.
        \item Three sets of $g$ non-separating, embedded curves $\{\alpha_i\}$, $\{\beta_i\}$, $\{\kappa_i\}$ on $\Sigma$ such that:
        \begin{enumerate}[(i)]
            \item All curves from a single set are disjoint, i.e. $\alpha_i\cap \alpha_j$ is nonempty when $i\neq j$.
            \item There exist no points $p\in \Sigma$ which are triple points with respect to the curve sets.
            \item Any pair of the three curve sets form a Heegaard diagram for $\mathop{\hbox{\large$\#$}}_{i=1}^k S^1\times S^2$, for some $k$ independent of which two curve sets are used.
        \end{enumerate}
    \end{enumerate}
\end{dff}

Each trisection diagram $T$ yields a well-defined trisected closed, connected, orientable 4-manifold $X(T)$, and vice-versa, any closed, connected, oriented 4-manifold admits a trisection and trisection diagram.

We say that two trisection diagrams are \emph{(oriented) diffeomorphic} if there is an oriented diffeomorphism of the underlying surfaces which intertwines the curve sets. We can also define the \emph{connect sum} $T\# T'$ of two trisection diagrams $T$ and $T'$ as the connected sum of the underlying surfaces via a disk disjoint from any curve sets.

We define the \emph{trisection moves} on a trisection diagram $T$:
\begin{enumerate}[{Move} 1:]
    \item (Handle Slides) Given two distinct $\alpha$ curves $\alpha_0$ and $\alpha_1$ along with an arc $\gamma$ connected $\alpha_0$ to $\alpha_1$, one may alter $T$ to a new trisection $T'$ by replacing $\alpha_0$ by the handle slide $\alpha_0\#_\gamma \alpha_1$ of $\alpha_0$ over $\alpha_1$ via $\gamma$. Here $\alpha_0 \#_\gamma \alpha_1$ is defined as follows: Let $U\subset \Sigma$ be a ribbon neighborhood of $\alpha_0\cup \gamma \cup \alpha_1$. The boundary $\partial U$ then decomposes into three closed curves: a normal push-off of $\alpha_0$, a normal push-off of $\alpha_1$, and a third piece, which is precisely the handle-slide $\alpha_0\#_\gamma \alpha_1$. We may similarly perform handle slides for $\beta$ and $\kappa$ curves.
    \item (Stabilization) A stabilization $T'$ of $T$ is the trisection given by the connect sum $T\#T_{\mathrm{st}}$ where $T_{\mathrm{st}}$ is the genus 3 stabilized sphere trisection.
    \item (Destabilization) For a trisection $T$ which is diffeomorphic to $T\cong T'\# T_{\mathrm{st}}$, a destabilization is just the summand $T'$.
\end{enumerate}

\begin{figure}
    \centering
    \begin{tikzpicture}
        \draw[closed, ultra thick] (0,0) to[curve through = {(1,-1) .. (2,-0.5) .. (3,-1) .. (4,-0.5) .. (5,-1) .. (6,0) .. (5,1) .. (4,0.5) .. (3,1) .. (2,0.5) .. (1,1)}] (0,0);
        \draw[ultra thick] (0.85,0) ellipse (0.25 and 0.4); 
        \draw[ultra thick] (3,0) ellipse (0.25 and 0.4); 
        \draw[ultra thick] (5.15,0) ellipse (0.25 and 0.4);
        \draw[blue, ultra thick] (0.75,-0.375) arc(90:270:0.2 and 0.32);
        \draw[blue, ultra thick, dashed] (0.75,-0.375) arc(90:-90:0.2 and 0.32);
        \draw[green, ultra thick] (1,-0.375) arc(90:270:0.2 and 0.32);
        \draw[green, ultra thick, dashed] (1,-0.375) arc(90:-90:0.2 and 0.32);
        \draw[green, ultra thick] (2.9,-0.375) arc(90:270:0.2 and 0.31);
        \draw[green, ultra thick, dashed] (2.9,-0.375) arc(90:-90:0.2 and 0.31);
        \draw[red, ultra thick] (3.15,-0.375) arc(90:270:0.2 and 0.3);
        \draw[red, ultra thick, dashed] (3.15,-0.375) arc(90:-90:0.2 and 0.3);
        \draw[red, ultra thick] (5.05,-0.375) arc(90:270:0.2 and 0.32);
        \draw[red, ultra thick, dashed] (5.05,-0.375) arc(90:-90:0.2 and 0.32);
        \draw[blue, ultra thick] (5.3,-0.375) arc(90:270:0.2 and 0.32);
        \draw[blue, ultra thick, dashed] (5.3,-0.375) arc(90:-90:0.2 and 0.32);
        \draw[red, ultra thick] (0.85,0) ellipse (0.55 and 0.7); 
        \draw[blue, ultra thick] (3,0) ellipse (0.55 and 0.7); 
        \draw[green, ultra thick] (5.15,0) ellipse (0.55 and 0.7);
    \end{tikzpicture}
    \caption{$T_{\mathrm{st}}$}
    \label{fig:stabilize_diagram}
\end{figure}

\begin{thm}[\cite{GK16}]\label{kirbygay_thm}
    Any two trisection diagrams $T$ and $T'$ of a closed, connected, oriented 4-manifold $X$ are oriented diffeomorphic after a finite series of trisection moves and diagram isotopies are applied to $T$.
\end{thm}

Now, diagram isotopies are very broad, so we restrict ourselves to a small subset of isotopies, which end up being all we need. A \emph{two-point move} is sliding two curves from different curve sets past each other along a disk in the surface. The following is a two-point move with an $\alpha$ curve and a $\beta$ curve:
    
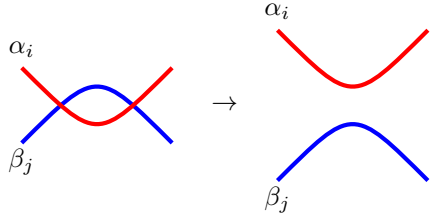
\begin{center}
        \begin{tikzpicture}
            \node at (0,0.75) {$\alpha_i$};
            \node at (0,-0.75) {$\beta_j$};
            \draw[blue,ultra thick] (0,-0.5) .. controls +(1,1) and +(-1,1) .. (2,-0.5);
            \draw[red,ultra thick] (0,0.5) .. controls +(1,-1) and +(-1,-1) .. (2,0.5);
            \node at (2.7,0) {$\rightarrow$};
            \begin{scope}[shift = {(3.4,0)}]
                \node at (0,1.25) {$\alpha_i$};
            \node at (0,-1.25) {$\beta_j$};
                \draw[blue, ultra thick] (0,-1) .. controls +(1,1) and +(-1,1) .. (2,-1);
                \draw[ red, ultra thick] (0,1) .. controls +(1,-1) and +(-1,-1) .. (2,1);
            \end{scope}
        \end{tikzpicture}
    \end{center}

    A \emph{three-point move} is sliding one curve past a crossing among the other two curve sets along a disk in the surface, as pictured below:
    \begin{center}
    \begin{tikzpicture}
        \node at (-2,0.75) {$\alpha_i$};
        \node at (-1.5,-1) {$\beta_j$};
        \node at (1.5,-1) {$\kappa_k$};
        \draw[blue, ultra thick] (-1,-1) -- (1,1);
        \draw[green, ultra thick] (1,-1) -- (-1,1);
        \draw[red, ultra thick] (-1.5,0.75) -- (1.5,0.75);
        \node at (2.75,0) {$\rightarrow$};
        \begin{scope}[shift = {(6,0)}]
        \node at (-2,-0.75) {$\alpha_i$};
        \node at (1.5,1) {$\beta_j$};
        \node at (-1.5,1) {$\kappa_k$};
        \draw[blue, ultra thick] (-1,-1) -- (1,1);
        \draw[green, ultra thick] (1,-1) -- (-1,1);
        \draw[red, ultra thick] (-1.5,-0.75) -- (1.5,-0.75);
        \end{scope}
    \end{tikzpicture}
    \end{center}

\begin{lem}[{\cite[Lemma 2.26]{CCC22}}]
    Let $\Gamma$ be a closed 1-manifold and $\Sigma$ a closed 2-manifold. Let $\iota_0,\iota_1:\Gamma\to \Sigma$ be homotopic immersions of $\Gamma$ in $\Sigma$ such that
    \begin{enumerate}[(a)]
        \item Each component $C$ of $\Gamma$ is embedded by $\iota_i$
        \item The components of $\iota_i(\Gamma)$ only intersect transversely at double points
    \end{enumerate}
    Then $\iota_0$ and $\iota_1$ are diffeomorphic after a series of two-point and three-point moves.
\end{lem}

Thus two isotopic trisection diagrams are oriented diffeomorphic after a finite sequence of two- and three-point moves.

Before we proceed, we must orient each curve in the curve sets $\alpha$, $\beta$, and $\kappa$, and we endow each curve with an arbitrary base point. We also fix a prescription for a positive crossing with each curve. This is (mostly) arbitrary, only changing pre-fixed data. We will use the prescription that for any pair of $\alpha_i$, $\beta_j$, $\kappa_k$ with intersection $p$, $(d_p(\alpha_i), d_p(\beta_j))$, $(d_p(\alpha_i),d_p(\kappa_k))$, and $(d_p(\kappa_k), d_p(\beta_j))$ each form an ordered basis for $T_p\Sigma$, where $d_p\gamma$ is the tangent vector at a point $p$ in the direction of $\gamma$. If this basis is agrees with the standard orientation of $\Sigma$, we call the crossing positive, otherwise, we call the crossing negative. Pictorially, the following crossings are positive when the plane is given the standard orientation:

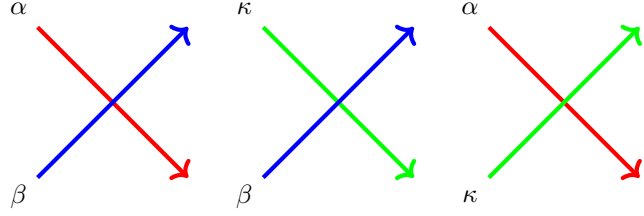
\begin{center}
    \begin{tikzpicture}
        \draw[red,ultra thick, ->] (-1,1) -- (1,-1);
        \draw[blue,ultra thick, ->] (-1,-1) -- (1,1);
        \node at (-1.25,-1.25) {$\beta$};
        \node at (-1.25,1.25) {$\alpha$};
        \draw[green,ultra thick, ->] (2,1) -- (4,-1);
        \draw[blue,ultra thick, ->] (2,-1) -- (4,1);
        \node at (1.75,-1.25) {$\beta$};
        \node at (1.75,1.25) {$\kappa$};
        \draw[red,ultra thick, ->] (5,1) -- (7,-1);
        \draw[green,ultra thick, ->] (5,-1) -- (7,1);
        \node at (4.75,-1.25) {$\kappa$};
        \node at (4.75,1.25) {$\alpha$};
    \end{tikzpicture}
\end{center}

Let $F[\alpha] = F[\alpha_1,\dots ,\alpha_g]$ denote the free group on the generators $\alpha_1$ to $\alpha_g$. Denote by $w_i^\beta(\alpha_1,\dots,\alpha_g)$ the word in $F[\alpha]$ given by $x_1\cdots x_{n_i}$, where $x_j$ corresponds to the $j$th crossing of an $\alpha$ curve encountered by traversing $\beta_i$ from its base point in the direction defined by its orientation. If this crossing is with $\alpha_k$, then define $x_j = \alpha_k^{\eta_k}$, where $\eta_k = 1$ if the crossing is positive, and $\eta_k = -1$ if the crossing is negative. Define $w_i^\kappa(\alpha_1,\dots,\alpha_g)$ in the same way as above, replacing any instance of $\beta$ with $\kappa$. We can now define a \emph{$G$-colored trisection diagram} as a trisection diagram where each $\alpha_i$ is colored by an element $a_i\in G$ such that $w_i^\beta(a_1,\dots ,a_g) = w_i^\kappa(a_1,\dots,a_g) = 1$.

Two $G$-colored trisection diagrams $T = (\Sigma, \alpha, \beta, \kappa)$ and $T' = (\Sigma', \alpha', \beta', \kappa')$ are \emph{equivalent} if one can be obtained from the other by a finite sequence of the following moves:

\begin{enumerate}
    \item (Surface Diffeomorphism) Performing an orientation-preserving diffeomorphism of one surface into the other, intertwining the curve sets and leaving the colors unchanged.
    \item (Orientation Reversal) Flipping the orientation of a curve. For an $\alpha$ curve $\alpha_i$, change its color $a_i$ to $a_i^{-1}$.
    \item (Isotopy of the Diagram) Performing an isotopy (two-point or three-point moves) of the curve sets, leaving the colors unchanged.

    \item (Stabilization/Destabilization) Remove a disk from $\Sigma$ away from the curve sets and performing the connected sum with the standard genus 3 trisection diagram for $S^4$. The added $\alpha$ curve is colored with $1\in G$. We may also perform the inverse procedure.

    \item (Handle Slide) Perform a handle slide trisection move to the underlying diagram. If $\alpha_i$ is slid over $\alpha_j$, assuming the connecting curve $\gamma$ is oriented from $\alpha_i$ to $\alpha_j$, and, without loss of generality, orient the curves relative to one another so that $(d_{\gamma(0)}\alpha_i, d_{\gamma(0)}\gamma)$ has the opposite parity to $(d_{\gamma(1)}\alpha_j, d_{\gamma(1)}\gamma)$ after parallel transport. Replace the color of $\alpha_j$ from $a_j$ to $a_i^{-1} a_j$. The following diagram illustrates this:
    
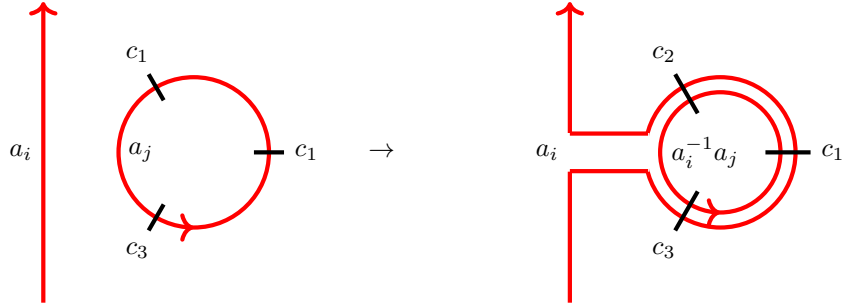
\begin{center}
    \begin{tikzpicture}
        \draw[<-,red, ultra thick] (-2,2) -- (-2,-2);
        \draw[red, ultra thick, decoration={markings, mark=at position 0.75 with {\arrow{>}}}, postaction={decorate}] (0,0) circle (1cm);
        \foreach \angle in {0,-120,120} {
            \draw[ultra thick] (\angle : 0.8) -- (\angle : 1.2);
        }
        \node at (0 : 1.5) {$c_1$};
        \node at (120 : 1.5) {$c_1$};
        \node at (-120 : 1.5) {$c_3$};
        \node at (-2.3,0) {$a_i$};
        \node at (-0.7,0) {$a_j$};
        \node at (2.5,0) {$\rightarrow$};
        \begin{scope}[shift = {(5,0)}]
        \draw[red, ultra thick, <-] (0,2) -- (0,-2);
        \begin{scope}[shift = {(2,0)}]
        \draw[red, ultra thick, decoration={markings, mark=at position 0.75 with {\arrow{>}}}, postaction={decorate}] (0,0) circle (0.8cm);
        \draw[red, ultra thick] (0,0) circle (1cm);
        \foreach \angle in {0,-120,120} {
            \draw[ultra thick] (\angle : 0.6) -- (\angle : 1.2);
        }
        \node at (0 : 1.5) {$c_1$};
        \node at (120 : 1.5) {$c_2$};
        \node at (-120 : 1.5) {$c_3$};
        \node at (-2.3,0) {$a_i$};
        \node at (-0.2,0) {$a_i^{-1}a_j$};
        \draw[fill = white,white] (-2.05,-0.25) rectangle (-0.95,0.25);
        \draw[red, ultra thick] (-2,-0.25) -- (-0.96,-0.25);
        \draw[red,ultra thick] (-2,0.25) -- (-0.96,0.25);
        \end{scope}
        \end{scope}
    \end{tikzpicture}
    \end{center}
\end{enumerate}

We remark that each of these does in fact take a $G$-colored trisection diagram to a $G$-colored trisection diagram.

\subsection{Trisection Bracket}

Let $G$ be a group with identity 1. Let $T = (\Sigma, \alpha,\beta,\kappa)$ be a $G$-colored trisection diagram such that the color assigned to each $\alpha_i$ is $a_i\in G$. Orient each circle arbitrarily and provide base points to each circle arbitrarily. Let $\mcal H = (H^\alpha,H^\beta,H^\kappa,\langle \_\rangle ^{\beta,\kappa},\langle \_\rangle^{\alpha,\beta}, \langle\_\rangle^{\kappa,\alpha})$ be an involutory Hopf $G$-triplet of finite type over the ground field $\mbb k$. Let $e^\alpha = \{e^\alpha_g\}_{g\in G}$ be a right $G$-cointegral for $H^\alpha$. Let $e^\beta$ and $e^\kappa$ be right cointegrals for the Hopf algebras $H^\beta_1$, $H^\kappa_1$, respectively.

To each $\alpha_i$ we assign the tensor
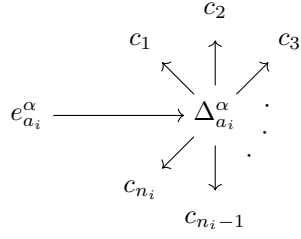
\begin{center}
    \begin{tikzpicture}
        \node at (0,0) (D) {$\Delta^\alpha_{a_i}$};
        \foreach \angle in {-135, -90,45,90,135} {
            \draw[->] (\angle:0.4) -- (\angle:1);
        }
        \node at (135:1.4) {$c_1$};
        \node at (90:1.4) {$c_2$};
        \node at (45:1.4) {$c_3$};
        \node at (-90:1.4) {$c_{n_i-1}$};
        \node at (-135:1.4) {$c_{n_i}$};
        \foreach \angle in {10, -20, -50} {
            \node at (\angle:0.7) {$\cdot$};
        }
        \node at (-2.5,0) (e) {$e^\alpha_{a_i}$};
        \draw[->] (e) -- (D);
    \end{tikzpicture}
\end{center}
where $c_k$ is the $k$th crossing encountered by traversing $\alpha_i$ from the base point along its orientation.

To each $\beta_j$, we assign the tensor
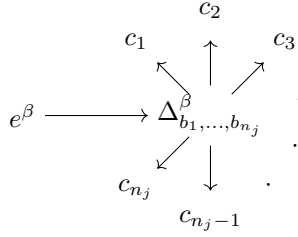
\begin{center}
    \begin{tikzpicture}
        \node at (0,0) (D) {$\Delta^\beta_{b_1,\dots ,b_{n_j}}$};
        \foreach \angle in {-135, -90,45,90,135} {
            \draw[->] (\angle:0.4) -- (\angle:1);
        }
        \node at (135:1.4) {$c_1$};
        \node at (90:1.4) {$c_2$};
        \node at (45:1.4) {$c_3$};
        \node at (-90:1.4) {$c_{n_j-1}$};
        \node at (-135:1.4) {$c_{n_j}$};
        \foreach \angle in {10, -20, -50} {
            \node at (\angle:1.2) {$\cdot$};
        }
        \node at (-2.5,0) (e) {$e^\beta$};
        \draw[->] (e) -- (D);
    \end{tikzpicture}
\end{center}
where $c_k$ is the $k$th crossing encountered by traversing $\beta_j$ from the base point along its orientation, and $b_k$ is defined as follows: $b_k = 1$ if $c_k$ is a crossing with a $\kappa$ curve, and if $c_k$ is a crossing of $\beta_j$ with $\alpha_l$, then $b_k = a_l^{\nu}$ where $\nu = 1$ if $c_k$ is a positive crossing, and $\nu = -1$ if $c_k$ is a negative crossing.

We define the tensor associated to $\kappa_p$ similarly:
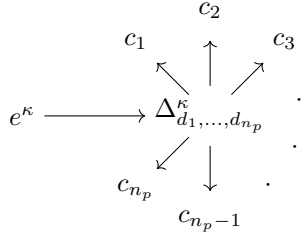
\begin{center}
    \begin{tikzpicture}
        \node at (0,0) (D) {$\Delta^\kappa_{d_1,\dots ,d_{n_p}}$};
        \foreach \angle in {-135, -90,45,90,135} {
            \draw[->] (\angle:0.4) -- (\angle:1);
        }
        \node at (135:1.4) {$c_1$};
        \node at (90:1.4) {$c_2$};
        \node at (45:1.4) {$c_3$};
        \node at (-90:1.4) {$c_{n_p-1}$};
        \node at (-135:1.4) {$c_{n_p}$};
        \foreach \angle in {10, -20, -50} {
            \node at (\angle:1.2) {$\cdot$};
        }
        \node at (-2.5,0) (e) {$e^\kappa$};
        \draw[->] (e) -- (D);
    \end{tikzpicture}
\end{center}
where $c_k$ is the $k$th crossing encountered by traversing $\kappa_m$ from the basepoint along its orientation, and $b_k$ is defined as follows: $d_k = 1$ if $c_k$ is a crossing with a $\beta$ curve, and if $c_k$ is a crossing of $\kappa_m$ with $\alpha_l$, then $d_k = a_l^{\nu}$ where $\nu = 1$ if $c_k$ is a positive crossing, and $\nu = -1$ if $c_k$ is a negative crossing.

To each crossing of $\alpha_i$ with $\beta_j$, we contract via the tensor
\begin{center}
    \begin{tikzpicture}
        \node at (0,0) (S) {$(S^\beta_{a_i^{(-1)^\nu}})^{\nu}$};
        \node at (2,0) (B) {$\bullet$};
        \draw[->] (-2,0) -- (S);
        \draw[->] (S) -- (B);
        \draw[<-] (B) -- (3,0);
    \end{tikzpicture}
\end{center}
where $\nu = 0$ if the crossing is positive and $\nu = 1$ if the crossing is negative.

To each crossing of $\alpha_i$ with $\kappa_j$, we contract via the tensor
\begin{center}
    \begin{tikzpicture}
        \node at (0,0) (S) {$(S^\kappa_{a_i^{(-1)^\nu}})^{\nu}$};
        \node at (2,0) (B) {$\bullet$};
        \draw[->] (-2,0) -- (S);
        \draw[->] (S) -- (B);
        \draw[<-] (B) -- (3,0);
    \end{tikzpicture}
\end{center}
where $\nu = 0$ if the crossing is positive and $\nu = 1$ if the crossing is negative.

To each crossing of $\beta_i$ with $\kappa_j$, we contract via the tensor
\begin{center}
    \begin{tikzpicture}
        \node at (0,0) (S) {$(S^\beta_1)^{\nu}$};
        \node at (2,0) (B) {$\bullet$};
        \draw[->] (-2,0) -- (S);
        \draw[->] (S) -- (B);
        \draw[<-] (B) -- (3,0);
    \end{tikzpicture}
\end{center}
where $\nu = 0$ if the crossing is positive and $\nu = 1$ if the crossing is negative.

Define the \emph{trisection bracket} $\langle T\rangle^G_{\mcal H, e}$ to be the contraction of this tensor network, where $e = \{e^\alpha, e^\beta, e^\kappa\}$.

Let $T_{\mathrm{st}}$ be the standard $G$-colored genus 3 stabilized sphere trisection diagram. Suppose that $z^3 =\langle T_{\mathrm{st}}\rangle^G_{ \mcal H, e}$ has a root $\zeta$ in the units $\mbb k^\times$ of the ground field. Set \[Z(T;G;\mcal H, e) = \zeta^{-g(\Sigma)}\langle T\rangle^G_{\mcal H, e}\] where $g(\Sigma)$ is the genus of the trisection surface.

\begin{thm}\label{coloredtrisectioninvariance_thm}
    $Z(T;G;\mcal H, e)$ is an equivalence invariant of $G$-colored trisection diagrams.
\end{thm}

\begin{proof}
    Suppose $T = (\Sigma, \alpha,\beta,\kappa)$ is colored by $(a_1,\dots ,a_g)$. We check each of the moves. For move 1, the tensors remain unchanged, so the bracket remains fixed. For orientation reversal, use the fact that $S^*_a$ is an anti-$G$-(co)algebra morphism for each $*\in \{\alpha,\beta,\kappa\}$.

    For a two-point move, a representative example has the following form:

    \begin{center}
        \begin{tikzpicture}
            \node at (0,0.75) {$\alpha_i$};
            \node at (0,-0.75) {$\beta_j$};
            \draw[->,ultra thick] (0,-0.5) .. controls +(1,1) and +(-1,1) .. (2,-0.5);
            \draw[<-,ultra thick] (0,0.5) .. controls +(1,-1) and +(-1,-1) .. (2,0.5);
            \node at (2.7,0) {$\rightarrow$};
            \begin{scope}[shift = {(3.4,0)}]
                \node at (0,1.25) {$\alpha_i$};
            \node at (0,-1.25) {$\beta_j$};
                \draw[->,ultra thick] (0,-1) .. controls +(1,1) and +(-1,1) .. (2,-1);
                \draw[<-,ultra thick] (0,1) .. controls +(1,-1) and +(-1,-1) .. (2,1);
            \end{scope}
        \end{tikzpicture}
    \end{center}

    The corresponding equalities demonstrate invariance:

    \begin{center}
        \begin{tikzpicture}
            \node at (-0.5,0) (Da) {$\Delta^\alpha_{a_i}$};
            \node at (3.25,0) (Db) {$\Delta^\beta_{a_i,a_i^{-1}}$};
            \node at (1.75,-0.5) (S) {$S^\beta_{a_i^{-1}}$};
            \node at (0.75,0.5) (BU) {$\bullet$};
            \node at (0.75,-0.5) (BB) {$\bullet$};
            \draw[->] (Da) -- (BU);
            \draw[->] (Da) -- (BB);
            \draw[->] (Db) -- (BU);
            \draw[->] (Db) -- (S);
            \draw[->] (S) -- (BB);
            \draw[<-] (Da) -- ++(-0.75,0);
            \draw[<-] (Db) -- ++(1,0);
            \node at (4.85,0) {$=$};
            \begin{scope}[shift = {(6.5,0)}]
            \node at (0.5,0) (M) {$M^\beta_{a_i}$};
            \node at (3.25,0) (D) {$\Delta^\beta_{a_i,a_i^{-1}}$};
            \node at (1.75,-0.5) (S) {$S^\beta_{a_i^{-1}}$};
            \node at (-0.5,0) (B) {$\bullet$};
            \draw[->] (D) .. controls +(-0.5,0.75) and +(0.5,0.75) .. (M);
            \draw[->] (D) -- (S);
            \draw[->] (S) -- (M);
            \draw[->] (M) -- (B);
            \draw[<-] (B) -- ++(-0.5,0);
            \draw[<-] (D) -- ++(1,0);
            \end{scope}
            \begin{scope}[shift = {(6.5,-1)}]
            \node at (-1.65,0) {$=$};
            \node at (0.5,0) (i) {$i^\beta_{a_i}$};
            \node at (1.5,0) (e) {$\epsilon^\beta$};
            \node at (-0.5,0) (B) {$\bullet$};
            \draw[<-] (B) -- ++(-0.5,0);
            \draw[<-] (e) -- ++(0.75,0);
            \draw[->] (i) -- (B);
            \end{scope}
            \begin{scope}[shift = {(6.5,-2)}]
            \node at (-1.65,0) {$=$};
            \node at (0.5,0) (i) {$\epsilon^\alpha_{a_i}$};
            \node at (1.5,0) (e) {$\epsilon^\beta$};
            \draw[<-] (e) -- ++(0.75,0);
            \draw[<-] (i) -- ++(-0.75,0);
            \end{scope}
        \end{tikzpicture}
    \end{center}

    For a three-point move, a representative example has the following form:

    
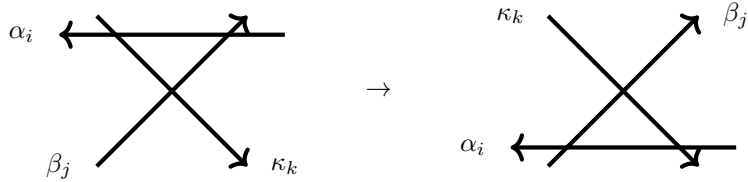
\begin{center}
    \begin{tikzpicture}
        \node at (-2,0.75) {$\alpha_i$};
        \node at (-1.5,-1) {$\beta_j$};
        \node at (1.5,-1) {$\kappa_k$};
        \draw[->,ultra thick] (-1,-1) -- (1,1);
        \draw[<-,ultra thick] (1,-1) -- (-1,1);
        \draw[<-,ultra thick] (-1.5,0.75) -- (1.5,0.75);
        \node at (2.75,0) {$\rightarrow$};
        \begin{scope}[shift = {(6,0)}]
        \node at (-2,-0.75) {$\alpha_i$};
        \node at (1.5,1) {$\beta_j$};
        \node at (-1.5,1) {$\kappa_k$};
        \draw[->,ultra thick] (-1,-1) -- (1,1);
        \draw[<-,ultra thick] (1,-1) -- (-1,1);
        \draw[<-,ultra thick] (-1.5,-0.75) -- (1.5,-0.75);
        \end{scope}
    \end{tikzpicture}
    \end{center}

    For invariance, the following equality must hold:

    \begin{center}
        \begin{tikzpicture}
            \node at (-1.5,-1.5) (DBL) {$\Delta^\kappa_{a_i,1}$};
            \node at (1.5,-1.5) (DBR) {$\Delta^\beta_{1,a_i^{-1}}$};
            \node at (0,-1.5) (BB) {$\bullet$};
            \node at (0,0.5) (DU) {$\Delta^\alpha_{a_i}$};
            \node at (-1.5,0.5) (BUL) {$\bullet$};
            \node at (1.5,0.5) (BUR) {$\bullet$};
            \node at (1.5,-0.5) (S) {$S^\beta_{a_i^{-1}}$};
            \draw[->] (DBR) -- (BB);
            \draw[->] (DBL) -- (BB);
            \draw[->] (DU) -- (BUL);
            \draw[->] (DU) -- (BUR);
            \draw[->] (DBL) -- (BUL);
            \draw[->] (DBR) -- (S);
            \draw[->] (S) -- (BUR);
            \draw[<-] (DU) -- ++(0,0.75);
            \draw[<-] (DBL) -- ++(0,-0.75);
            \draw[<-] (DBR) -- ++(0,-0.75);
            \node at (2.5,-0.75) {$=$};
            \begin{scope}[shift = {(5,0)}]
            \node at (1.5,0.5) (DUR) {$\Delta^{\kappa,cop}_{a_i^{-1},1}$};
            \node at (-1.5,0.5) (DUL) {$\Delta^{\beta,cop}_{1,a_i}$};
            \node at (0,0.5) (BU) {$\bullet$};
            \node at (0,-1.5) (DB) {$\Delta^{\alpha,cop}_{a_i}$};
            \node at (-1.5,-1.5) (BBL) {$\bullet$};
            \node at (1.5,-1.5) (BBR) {$\bullet$};
            \node at (-1.5,-0.5) (S) {$S^\beta_{a_i^{-1}}$};
            \draw[->] (DB) -- (BBR);
            \draw[->] (DB) -- (BBL);
            \draw[->] (DUL) -- (S);
            \draw[->] (S) -- (BBL);
            \draw[->] (DUL) -- (BU);
            \draw[->] (DUR) -- (BBR);
            \draw[->] (DUR) -- (BU);
            \draw[<-] (DB) -- ++(0,-0.75);
            \draw[<-] (DUR) -- ++(0,0.75);
            \draw[<-] (DUL) -- ++(0,0.75);
            \end{scope}
        \end{tikzpicture}
    \end{center}

    This equality follows from Lemma \ref{fundlemoftriplets_lem} by taking the equality in (b) and pushing the antipodes through, using that the bilinear forms induce morphisms of Hopf $G$-algebras (or Hopf algebras when applicable). One can similarly use Lemma \ref{fundlemoftriplets_lem} to prove the other versions of the three-point move.

    It is clear that performing the connected sum of trisection diagrams results in the product of trisection brackets. Due to this, stabilization will multiply the bracket by $\langle T_{\mathrm{st}}\rangle$, and adds 3 to the genus of the trisection surface. This clearly cancels with the change in the normalization factor $\zeta^{-g(\Sigma)}$.

    Finally, for handle slide. A representative example has the following form:

    
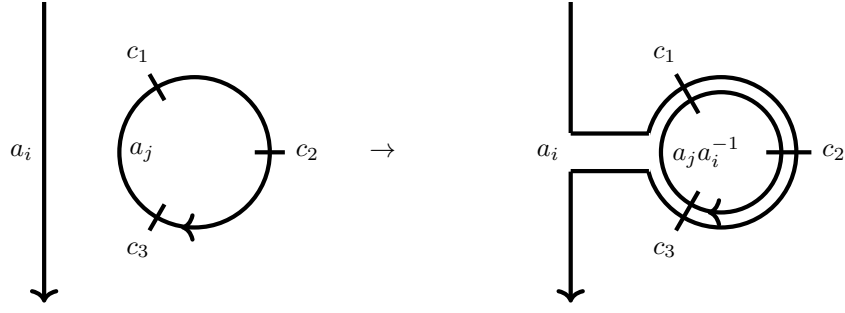
\begin{center}
    \begin{tikzpicture}
        \draw[->,ultra thick] (-2,2) -- (-2,-2);
        \draw[ultra thick,decoration={markings, mark=at position 0.75 with {\arrow{<}}}, postaction={decorate}] (0,0) circle (1cm);
        \foreach \angle in {0,-120,120} {
            \draw[ultra thick] (\angle : 0.8) -- (\angle : 1.2);
        }
        \node at (0 : 1.5) {$c_2$};
        \node at (120 : 1.5) {$c_1$};
        \node at (-120 : 1.5) {$c_3$};
        \node at (-2.3,0) {$a_i$};
        \node at (-0.7,0) {$a_j$};
        \node at (2.5,0) {$\rightarrow$};
        \begin{scope}[shift = {(5,0)}]
        \draw[->,ultra thick] (0,2) -- (0,-2);
        \begin{scope}[shift = {(2,0)}]
        \draw[ultra thick,decoration={markings, mark=at position 0.75 with {\arrow{<}}}, postaction={decorate}] (0,0) circle (0.8cm);
        \draw[ultra thick] (0,0) circle (1cm);
        \foreach \angle in {0,-120,120} {
            \draw[ultra thick] (\angle : 0.6) -- (\angle : 1.2);
        }
        \node at (0 : 1.5) {$c_2$};
        \node at (120 : 1.5) {$c_1$};
        \node at (-120 : 1.5) {$c_3$};
        \node at (-2.3,0) {$a_i$};
        \node at (-0.2,0) {$a_ja_i^{-1}$};
        \draw[fill = white,white] (-2.05,-0.25) rectangle (-0.95,0.25);
        \draw[ultra thick] (-2,-0.25) -- (-0.96,-0.25);
        \draw[ultra thick] (-2,0.25) -- (-0.96,0.25);
        \end{scope}
        \end{scope}
    \end{tikzpicture}
    \end{center}

    The following equalities demonstrates invariance: 

    
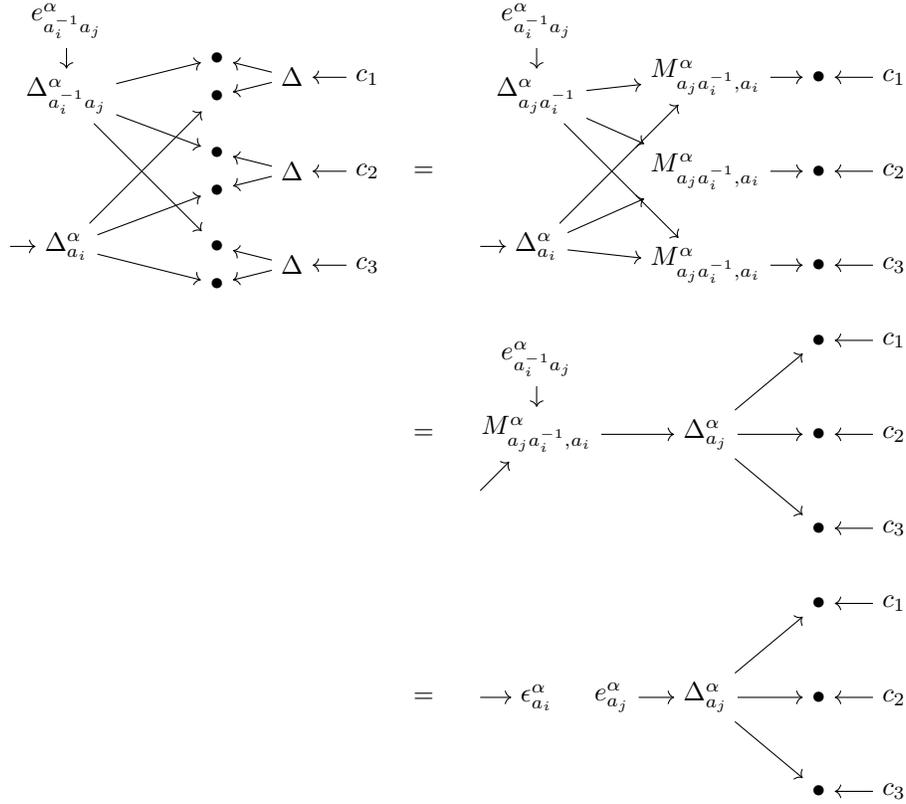
\begin{center}
        \begin{tikzpicture}
            \node at (0,1) (e) {$e^\alpha_{a_i^{-1}a_j}$};
            \node at (0,0) (DU) {$\Delta^\alpha_{a_i^{-1} a_j}$};
            \node at (0,-2) (DL) {$\Delta^\alpha_{a_i}$};
            \node at (2,0.5) (BU1) {$\bullet$};
            \node at (2,0) (BU2) {$\bullet$};
            \node at (2,-0.75) (BM1) {$\bullet$};
            \node at (2,-1.25) (BM2) {$\bullet$};
            \node at (2,-2) (BL1) {$\bullet$};
            \node at (2,-2.5) (BL2) {$\bullet$};
            \node at (3,0.25) (DUR) {$\Delta$};
            \node at (3,-1) (DMR) {$\Delta$};
            \node at (3,-2.25) (DBR) {$\Delta$};
            \node at (4,0.25) (c1) {$c_1$};
            \node at (4,-1) (c2) {$c_2$};
            \node at (4,-2.25) (c3) {$c_3$};
            \draw[->] (DU) -- (BU1);
            \draw[->] (DU) -- (BM1);
            \draw[->] (DU) -- (BL1);
            \draw[->] (DL) -- (BU2);
            \draw[->] (DL) -- (BM2);
            \draw[->] (DL) -- (BL2);
            \draw[->] (DUR) -- (BU1);
            \draw[->] (DUR) -- (BU2);
            \draw[->] (DMR) -- (BM1);
            \draw[->] (DMR) -- (BM2);
            \draw[->] (DBR) -- (BL1);
            \draw[->] (DBR) -- (BL2);
            \draw[<-] (DU) -- (e);
            \draw[<-] (DL) -- ++(-0.75,0);
            \draw[<-] (DUR) -- (c1);
            \draw[<-] (DMR) -- (c2);
            \draw[<-] (DBR) -- (c3);
            \node at (4.75,-1) {$=$};
            \begin{scope}[shift = {(6,0)}]
            \node at (0.25,1) (e) {$e^\alpha_{a_i^{-1}a_j}$};
                \node at (0.25,0) (DU) {$\Delta^\alpha_{a_ja_i^{-1}}$};
            \node at (0.25,-2) (DL) {$\Delta^\alpha_{a_i}$};
            \node at (2.5,0.25) (MU) {$M^\alpha_{a_ja_i^{-1},a_i}$};
            \node at (2.5,-1) (MM) {$M^\alpha_{a_ja_i^{-1},a_i}$};
            \node at (2.5,-2.25) (ML) {$M^\alpha_{a_ja_i^{-1},a_i}$};
            \node at (4,0.25) (BU) {$\bullet$};
            \node at (4,-1) (BM) {$\bullet$};
            \node at (4,-2.25) (BL) {$\bullet$};
            \node at (5,0.25) (c1) {$c_1$};
            \node at (5,-1) (c2) {$c_2$};
            \node at (5,-2.25) (c3) {$c_3$};
            \draw[<-] (DU) -- (e);
            \draw[<-] (DL) -- ++(-0.75,0);
            \draw[<-] (BU) -- (c1);
            \draw[<-] (BM) -- (c2);
            \draw[<-] (BL) -- (c3);
            \draw[<-] (BU) -- (MU);
            \draw[<-] (BM) -- (MM);
            \draw[<-] (BL) -- (ML);
            \draw[->] (DU) -- (MU);
            \draw[->] (DU) -- (MM);
            \draw[->] (DU) -- (ML);
            \draw[->] (DL) -- (MU);
            \draw[->] (DL) -- (MM);
            \draw[->] (DL) -- (ML);
            \end{scope}
            \begin{scope}[shift = {(6,-3.5)}]
            \node at (-1.25,-1) {$=$};
            \node at (0.25,0) (e) {$e^\alpha_{a_i^{-1}a_j}$};
            \node at (0.25,-1) (M) {$M^\alpha_{a_ja_i^{-1},a_i}$};
            \node at (2.5,-1) (D) {$\Delta^\alpha_{a_j}$};
            \node at (4,0.25) (BU) {$\bullet$};
            \node at (4,-1) (BM) {$\bullet$};
            \node at (4,-2.25) (BL) {$\bullet$};
            \node at (5,0.25) (c1) {$c_1$};
            \node at (5,-1) (c2) {$c_2$};
            \node at (5,-2.25) (c3) {$c_3$};
            \draw[<-] (M) -- (e);
            \draw[<-] (M) -- ++(-0.75,-0.75);
            \draw[<-] (BU) -- (c1);
            \draw[<-] (BM) -- (c2);
            \draw[<-] (BL) -- (c3);
            \draw[->] (M) -- (D);
            \draw[->] (D) -- (BU);
            \draw[->] (D) -- (BM);
            \draw[->] (D) -- (BL);
            \end{scope}
            \begin{scope}[shift = {(6,-7)}]
            \node at (-1.25,-1) {$=$};
            \node at (1.25,-1) (e) {$e^\alpha_{a_j}$};
            \node at (0.25,-1) (eps) {$\epsilon^\alpha_{a_i}$};
            \node at (2.5,-1) (D) {$\Delta^\alpha_{a_j}$};
            \node at (4,0.25) (BU) {$\bullet$};
            \node at (4,-1) (BM) {$\bullet$};
            \node at (4,-2.25) (BL) {$\bullet$};
            \node at (5,0.25) (c1) {$c_1$};
            \node at (5,-1) (c2) {$c_2$};
            \node at (5,-2.25) (c3) {$c_3$};
            \draw[<-] (eps) -- ++(-0.75,0);
            \draw[<-] (BU) -- (c1);
            \draw[<-] (BM) -- (c2);
            \draw[<-] (BL) -- (c3);
            \draw[->] (e) -- (D);
            \draw[->] (D) -- (BU);
            \draw[->] (D) -- (BM);
            \draw[->] (D) -- (BL);
            \end{scope}
        \end{tikzpicture}
    \end{center}
    where here we use an unmarked $\Delta$ to mean $\Delta^\beta_b$ or $\Delta^\kappa_b$ for some $b\in G$ which makes sense in its context.
\end{proof}

\begin{lem}\label{crossedinvariance_lem}
    If $H^\alpha$ is crossed, then $Z(T;G;\mcal H, e)$ does not depend on the conjugacy class of the colors of the $T$.
\end{lem}

\begin{proof}
    Suppose $T = (\Sigma,\alpha,\beta,\kappa)$ is a $G$-colored trisection diagram with color $(a_1,\dots, a_g)$. Denote by $T^b$ the conjugated $G$-colored trisection diagram with color $(ba_1 b^{-1},\dots, ba_g b^{-1})$.

    Suppose $H^\alpha$ admits the crossing $\phi = \{\phi_h : H^\alpha_a \to H^\alpha_{ha h^{-1}}\}_{h\in G}$. Now, since $H^\alpha$ is cosemisimple, we have that $\phi_h(e^\alpha_a) = e^\alpha_{hah^{-1}}$, indeed, $H^\alpha$ is cosemisimple if and only if $\epsilon_a(e^\alpha_a) = 1$ for all $a$ in $G$ where $H^\alpha_a\neq 0$, thus $\epsilon_{hah^{-1}}(e^\alpha_{hah^{-1}}) = 1 = \epsilon_{hah^{-1}}(\phi_h(e^\alpha_a))$, and since $\phi_h(e^\alpha_a)$ and $e^\alpha_{hah^{-1}}$ are nonzero cointegrals, they differ by a nonzero scalar (whenever $H^\alpha_a\neq 0$). The above then shows $e^\alpha_{hah^{-1}} = \phi_h\circ e^\alpha_a$.

    This, along with the fact that $\phi_b \circ \phi_{b^{-1}} = id$ and the other axioms for a crossing, give that $\langle T\rangle_{\mcal H,e}^G = \langle T^b\rangle_{\mcal H,e}^G$.
\end{proof}

\section{Invariants of Flat Bundles over 4-Manifolds}
\label{FlatBundlesInvariants_Sec}
\subsection{Flat Bundles over 4-Manifolds}

Let $G$ be a group. A \emph{flat $G$-bundle over a 4-manifold} is a principal $G$-bundle $\xi = (p:\tilde X\to X)$, where $X$ is a closed connected and oriented 4-manifold, which is flat, i.e. its transition functions are locally constant. $\tilde X$ is called the \emph{total space} and $X$ is called the \emph{base space} of $\xi$.

Two flat $G$-bundles $\xi$ and $\xi'$ are \emph{equivalent} if there exists a diffeomorphism $\tilde h:\tilde X\to \tilde X'$ between total spaces which preserves the action of $G$ and which induces an orientation-preserving diffeomorphism $h:X\to X'$ between their base spaces.

A flat $G$-bundle $\xi$ is \emph{pointed} if its total space $\tilde X$ is endowed with a base point $\tilde x$. Two pointed flat $G$-bundles are \emph{equivalent} if there is an equivalence between them preserving the base point. It is well-known that pointed flat bundles of 4-manifolds are in bijective correspondence with homomorphisms $\pi_1(X,x)\to G$, and if one ignores the basepoint, this becomes a correspondence with conjugacy classes.

The following presentation of the fundamental group of a closed connected orientable 4-manifold is well known.

\begin{lem}\label{fundgroupfromdiagram_lem}
Let $X$ be a closed oriented connected 4-manifold with trisection diagram $(\Sigma, \alpha,\beta,\kappa)$. Suppose $X$ has base point $x$ lying in $\Sigma$ away from the curve sets. Denote by $\hat{\alpha}_i$ simple closed curves in $\Sigma$ such that $\hat{\alpha}_i$ intersects $\alpha$ at $\alpha_i$ in one point, and intersects $\alpha_{i'}$ at no points whenever $i' \neq i$. Then
\[\pi_1(X,x) = \langle \{[\hat{\alpha}_i]\}_{i=1}^g \mid \{w_i^\beta([\hat{\alpha}_1],\dots [\hat{\alpha}_g]) = w_j^\kappa([\hat{\alpha}_1],\dots, [\hat{\alpha}_g]) = 1\}_{i,j = 1 }^g\rangle\] where $w_i^\beta(\alpha_1,\dots,\alpha_g)$ and $w_j^\kappa(\alpha_1,\dots,\alpha_g)$ are defined as the words in $F[\alpha]$ as in section \ref{ColoredTrisectionInvariants_Sec}.
\end{lem}

\subsection{The Invariant}

Let $G$ be a group with identity $1$. Let $(\xi,\tilde x)$ be a pointed flat $G$-bundle over the 4-manifold $X$. Let $f:\pi_1(X,x)\to G$ be the monodromy, and set $a_i = f([\hat{\alpha}_i])$. This is a color for the trisection diagram $T$ by lemma \ref{fundgroupfromdiagram_lem}, and we say that $T$ is \emph{colored by $f$}. Set $Z((\xi,\tilde x);G;\mcal H,e) = Z(T;G;\mcal H, e)$. 

The following theorem mimics that of \cite[Theorem 2]{V05}, generalized to 4-manifolds:
\begin{thm}\label{thm:main}
    \begin{enumerate}[(a)]
    \item $Z((\xi,\tilde x);G;\mcal H, e)$ is an invariant of pointed flat bundles over 4-manifolds.
    \item The function taking $\tilde x\in \tilde X \mapsto Z((\xi,\tilde x);G;\mcal H,e)$ is constant on the path components of $\tilde X$. 
    \item If $H^\alpha$ is crossed, or $G$ is abelian, or the monodromy of $\xi$ is surjective, then $Z((\xi,\tilde x);G;\mcal H,e)$ does not depend on the base point $\tilde x$.
    \end{enumerate}
\end{thm}

\begin{proof}
    Let us begin by proving part (a). Let $(X, x,f)$ and $(X',x',f')$ determine two equivalent pointed flat $G$-bundles over 4-manifolds. Let $T$ (resp. T') be an oriented trisection diagram for $X$ (resp. $X'$) colored by $f$ (resp. $f'$). By theorem \ref{coloredtrisectioninvariance_thm}, we just need to show that $T$ and $T'$ are equivalent $G$-colored trisection diagrams.

    Since $(X,x,f)$ and $(X',x',f')$ are equivalent, there exists a pointed orientation-preserving diffeomorphism $h:X\to X'$ with $f = f'\circ h_*$, where $h_*:\pi_1(X,x) \to \pi_1(X',x')$ is the induced map. By theorem \ref{kirbygay_thm}, $T$ and $T'$ are related by a finite series of trisection moves and isotopies. We orient the trisection curves arbitrarily and note that we may assume the orientation-preserving diffeomorphism $h$ preserves the orientation of the circles, since swapping the orientation amounts to changing the color by its inverse. Without loss of generality, assume that $T$ and $T'$ differ by only one trisection move, then denote the genus of $T$ (resp. $T'$) by $g$ (resp. $g'$) and denote the color of $T$ (resp. $T'$) by $a = (a_1, \dots ,a_g)$ (resp. $a' = (a_1', \dots ,a'_{g'})$.

    For a trisection move of type 1 (handle slides), we first assume that $\beta_i$ (resp. $\kappa_i$) is slid over $\beta_j$ (resp. $\kappa_j$). Note that since the handle slide occurs along a band which is disjoint from the trisection curves and we can assume it is disjoint from the basepoint, the result of this slide amounts to inserting a word $w_j^\beta$ or $(w_j^{\beta})^{-1}$ into $w_i^\beta$. Since $w_j^\beta = 1$, the $G$-colored diagrams are clearly equivalent. Now assume that $\alpha_i$ is slid over $\alpha_j$ away from the basepoint. Let $\gamma_i$ (resp. $\gamma_j$) be a loop based at $x$ before the handle slide intersecting $\alpha_i$ (resp. $\alpha_j$) positively at one point and intersecting no other $\alpha$ circles. After the handle slide, set $\gamma_i' = h(\gamma_i)$ and set $\gamma_j'$ to be a loop comprised of a path from $x'$ through the core of the sliding band then back to $x'$ which intersects $\alpha_j$ once positively, and does not intersect any other $\alpha$ curve. Then \[a_i' = f'(\gamma_i') = f'\circ h(\gamma_i) = f(\gamma_i) = a_i\]
    Now, $\gamma_j'$ is homotopic to the loop $h(\gamma_i)^{-1}h(\gamma_j)$, and thus 
    \begin{align*}
        a_j' &= f'(\gamma_j')\\
        &= f'(h(\gamma_i)^{-1}h(\gamma_j))\\
        &= f(\gamma_i)^{-1}f(\gamma_j)\\
        &= a_i^{-1}a_j
    \end{align*}

    For a move of type 2 (stabilization), denote the new $\alpha'$ circles by $\{C^1_\alpha,C^2_\alpha,C^3_\alpha\} \in \alpha'$. Notice that $\alpha' = h(\alpha) \cup \{C^i_\alpha\}_{i=1}^3$, where the color of $\alpha_i'$ agrees with $\alpha_i$ whenever $i\in \{1,\dots ,g\}$. Let $C_i$ be the ($\beta$ or $\kappa$ circle) which intersects $C^i_\alpha$ at exactly one point in the genus 3 trisection diagram for $S^4$. Let $\gamma_i$ be a path based at $x'$ which has endpoint in $C_i$, intersecting positively, and does not intersect any other $\alpha'$ circles or $C_i$ at any other point. Then the path $\gamma_iC_i\gamma_i^{-1}$ is a loop based at $x'$ which intersects $C^i_\alpha$ at exactly one point (WLOG positively) and does not intersect any other $\alpha'$ circles. Since $C_i$ bounds a disk, this loop is trivial, and hence each $C^i_\alpha$ is colored by 1. 

    A move of type 3 is the inverse of type 2, hence we omit the discussion.

    For an isotopy of the trisection diagram we note that this does not change the homotopy class of any generators of $\pi_1(X,x)$, and hence all colors will remain the same. Recall that $T$ and $T'$ are diffeomorphic after a sequence of two- and three-point moves. In particular, isotopies of the diagrams are generated by two- and three-point moves.

    This entire discussion proves that $T$ and $T'$ differ by a finite sequence of $G$-colored trisection diagram moves and hence are equivalent $G$-colored trisection diagrams. This proves part (a).

    We move on to prove part (b). Note that if $\tilde x$ and $\tilde x'$ lie in the same path component, then pushing $\tilde x_1$ to $\tilde x_2$ along a path connecting the two inside a tubular neighborhood of the path yields an equivalence of pointed flat $G$-bundles. Hence the invariants are the same.

    For part (c), let $\tilde x$ and $\tilde x'$ be two points in the total space of a flat bundle $\xi = (p:\tilde X\to X)$. Let $x = p(\tilde x)$ and $x' = p(\tilde x')$ and let $\gamma$ be a path from $x$ to $x'$ in $X$. Lift $\gamma$ to $\tilde \gamma$ and let $\tilde z$ be its endpoint. Then $p(\tilde z) = p(\tilde x') = x'$, so there exists an element $b\in G$ with $\tilde x' = b\cdot \tilde z$. Suppose the monodromy of $\xi$ is surjective, then there exists a loop $\sigma$ based at $x'$ such that $f([\sigma]) = b$. Let $\tilde \sigma$ be its lift starting at $\tilde z$. Then since $\tilde x' = b\cdot \tilde z = f([\sigma]) \cdot \tilde z$, then $\tilde \sigma$ must be a path from $\tilde z$ to $\tilde x'$. Hence $\tilde \sigma\circ \tilde\gamma$ is a path from $\tilde x$ to $\tilde x'$, so by part (b), the invariants are the same.

    Now suppose $H^\alpha$ is crossed. Since $\tilde x' = b\cdot \tilde z$, the monodromies $f_{\tilde x'}, f_{\tilde z}:\pi_1(X,x)\to G$ determined by $\tilde x'$ and $\tilde z$, respectively, are conjugate, so that $f_{\tilde x'} = bf_{\tilde z}b^{-1}$. Let $T$ be a trisection diagram for $X$ and denote by $T_{\tilde x'}$ and $T_{\tilde z}$ the colored trisection diagrams colored by $f_{\tilde x'}$ and $f_{\tilde z}$, respectively. Since the monodromies are conjugate, the colors of the diagrams are conjugate as well. Hence by lemma \ref{crossedinvariance_lem} and part (a), the invariants are the same.

    If $G$ is abelian, then $H^\alpha$ is crossed, so we are done.
\end{proof}

\begin{rmk}
    Note that if $G$ is finite, then we can obtain an invariant of 4-manifolds on the nose by summing over equivalence class of pointed bundles \[Z(X;G;\mcal H, e) = \sum_{\text{bundles }(\xi,\tilde x)} Z((\xi,\tilde x);G;\mcal H, e)\]
\end{rmk}

\section{Relationship to Mochida's Invariant}
\label{MochidaInvariant_Sec}
We begin this section by recalling the invariant of flat bundles of 4-manifolds defined by Mochida in \cite{Moc24}.

Let $\mcal H = (H,\phi,R)$ be an involutory quasitriangular Hopf $G$-algebra of finite type, where $\phi$ is a crossing on $H$ and $R$ is a universal $R$-matrix in $H_1\otimes H_1$. Let $e = \{e_a\}_{a\in G}$ be a two-sided cointegral and $\mu$ a two-sided integral such that $\mu(e_1) = 1$ and $\epsilon_1(e_1) = 1$. Let $f:\pi_1(X,x)\to G$ be the monodromy describing a pointed flat bundle for a 4-manifold $X$ and let $L = L_1\cup L_2$ be a Kirby diagram of $X$ with $L_1$ the dotted circles and $L_2$ the undotted circles. By a coloring of $L$ we mean an assignment of an element of $G$ to each dotted circle in $L_1$ satisfying the relations outlined in \cite[Section 3.1.1]{Moc24}. Then we define an invariant of the bundle as follows:

For each dotted circle $\alpha$ of $L_1$ colored by $a\in G$, let $q_1, \dots, q_k$ be the intersections of the disk bounded by $\alpha$ with the undotted components passing through it, listed from left to right. Assign to them the tensor \[S_a^{\sigma_{q_1}}(e_a^{(1)})\otimes \cdots \otimes S_a^{\sigma_{q_k}}(e_a^{(k)})\] so that the $i$-th component of the tensor is assigned to $q_i$. If there are no undotted components passing through $\alpha$, then the contribution is $\epsilon_a(e_a)$, which is 1 if $H_a\neq 0$ and 0 if $H_a = 0$.

For each crossing of the undotted components, let $q_1$ and $q_2$ be points on the over and under strands near the crossing, respectively. If the crossing is positive, we assign the universal $R$-matrix $R = s_i \otimes t_i\in H_1\otimes H_1$ so that the first component is assigned to $q_1$ and the second to $q_2$. If the crossing is negative, we assign $R^{-1} = S(s_i)\otimes t_i$ in the same manner as in the positive setting.

We evaluate the product of these assigned elements on each undotted component by the integral $\mu$ and define $\langle L,f\rangle_{\mcal H}$ to be the product of the evaluations for all undotted components. We then obtain an invariant $I_{\mcal H}(X,f,L)$ by normalizing: \[I_\mcal{H}(X,f,L) = \dim(H_1)^{|L_1|-|L_2|} \langle L,f\rangle_{\mcal H}\]

\begin{thm}
    Let $\mcal H$ be the Hopf $G$-triplet defined as in lemma \ref{Mochida_Triplet_Lem} and let $\mu$ and $\{e_a\}_{a\in G}$ be a two-sided integral for $H_1$ and a two-sided cointegral for $H$, respectively, such that $\mu(e_1) = 1$ and $\epsilon_1(e_1) = 1$. Then for $(\xi,\tilde x)$ a flat bundle over the 4-manifold $X$ with monodromy $f$, \[Z((\xi,\tilde x);G;\mcal H,e) = I_{\mcal H}(X,f,L)\] where $L$ is some Kirby diagram for $X$, and $I_{\mcal H}(X,f,L)$ is the invariant of flat bundles described by \cite{Moc24}.
\end{thm}

\begin{proof}
    We begin this proof by reviewing two algorithms between trisection diagrams and Kirby diagrams. These are described in \cite{Kep22} and are reviewed here:

    \begin{enumerate}[{Algorithm} 1:]
        \item Given a Kirby diagram describing a closed oriented 4-manifold $X$, which comprises $k$ 1-handles and a framed link $L$ consisting of $l$ knots $L_1,\dots ,L_l$, one obtains a trisection diagram describing $X$ by applying the following algorithm:
        \begin{enumerate}
            \item Replace the projection of the attaching regions of the 4-dimensional 1-handles by attaching regions of 3-dimensional 1-handles (a pair of 2 disks) with a parallel $\alpha$-$\beta$ curve pair drawn around one of the disks.
            \item For every knot in the link $L$ that does not have an overcrossing, introduce a $+1$- then a $-1$-kink. In other words we do a Reidemeister-2 move between two strands of the same knot.
            \item Introduce a number of $\pm 1$ kinks equal to the sign and numerical value of the blackboard framing.
            \item Replace all intersections with a cancelling $\alpha$-$\beta$ pair where the overcrossing strand intersects only the $\beta$ curve only once and the undercrossing strand intersects only the $\alpha$ curve only once. Note that this increases the number of $\alpha$ and $\beta$ curves from $l$ to some $g\in \mbb N$.
            \item Do handle slides among the $\alpha$-curves and rename them such that $\alpha_i$ intersects only $L_i$ only once.
            \item Rename $L_j$ to $\kappa_j$ and color them green.
            \item Draw a green curve parallel to every unpaired $\alpha$ curve and call these new green curves $\kappa_{l+1},\dots,\kappa_g$.
        \end{enumerate}

        \item Let $(\Sigma_g,\alpha,\beta,\kappa)$ be a trisection diagram of a closed oriented 4-manifold $X$ in $(\alpha,\beta)$-standard form (i.e. $(\Sigma_g,\alpha,\beta)$ forms a $(g-k)$-stabilized standard Heegaard diagram for $\#^k S^1\times S^2$, where a standard Heegaard diagram for $S^1\times S^2$ is a parallel pair of $\alpha$-$\beta$ curves. Then one obtains a Kirby diagram describing $X$ by applying the following algorithm:
        \begin{enumerate}
            \item Replace every parallel $\alpha$-$\beta$ pair with a dotted circle.
            \item Delete all other $\alpha$-$\beta$ pairs.
            \item Regard all $\kappa$ curves as a framed link, where the framing is induced by the surface $\Sigma$.
            \item Delete the surface $\Sigma_g$.
        \end{enumerate}
    \end{enumerate}

    Our interest lies mainly with algorithm 1. Let $L$ be a colored Kirby diagram representing the flat bundle $(\xi,\tilde x)$. Let $T(L)$ be the (not necessarily trisection) diagram obtained by performing only steps (a) through (d) of algorithm 1. Note that there is a canonical way to color $T(L)$ so that finishing algorithm 1 yields a flat bundle which is equivalent to the flat bundle obtained by coloring $K$ by the monodromy.
    
    The following picture demonstrates a positive link crossing in $T(L)$:

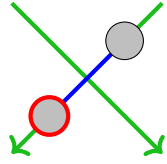
    \begin{center}
        \begin{tikzpicture}
            \draw[->,darkgreen,ultra thick] (-1,1) -- (1,-1);
            \draw[darkgreen,ultra thick] (1,1) -- (0.5,0.5);
            \draw[darkgreen,->,ultra thick] (-0.5,-0.5) -- (-1,-1);
            \draw[blue,ultra thick] (0.5,0.5) -- (-0.5,-0.5);
            \draw[draw=black, fill=lightgray] (0.5,0.5) circle (0.25);
            \draw[draw=black, fill=lightgray] (-0.5,-0.5) circle (0.25);
            \draw[draw=red,ultra thick] (-0.5,-0.5) circle (0.25);
        \end{tikzpicture}
    \end{center}

    The following tensor is assigned to this picture:
    
    \begin{center}
        \begin{tikzpicture}
            \node at (0.5,0.5) (M) {$M_{1,1}$};
            \node at (1.5,0.5) (mu) {$\mu$};
            \node at (-1,0.5) (R) {$R^{-1}$};
            \node at (-1,-1) (D) {$\Delta_1$};
            \node at (0,-1) (e) {$e$};
            \draw[->] (M) -- (mu);
            \draw[<-] (M) -- (D);
            \draw[<-] (M) -- (R);
            \draw[->] (R) -- ++(-1,0);
            \draw[->] (D) --  ++(-1,0);
            \draw[->] (e) -- (D);
            \node at (-1.5,0.75) {\tiny{1}};
            \node at (2,0) {$=$};
            \begin{scope}[shift = {(4,0)}]
            \node at (0,0) (R) {$R$};
            \draw[->] (R) -- ++(-1,0.5);
            \draw[->] (R) -- ++(-1,-0.5);
            \node at (-0.5,0.5) {\tiny{1}};
            \end{scope}
        \end{tikzpicture}
    \end{center}

    Hence the R-matrix is assigned to a crossing, with the first leg of $R$ assigned to the overcrossing strand. For a negative crossing of the Kirby link, the tensor is given by the inverse R-matrix $R^{-1}$, with the first leg assigned to the overcrossing. For a parallel pair of $\alpha$ and $\beta$ curves, we have the following representative picture in $T(L)$:

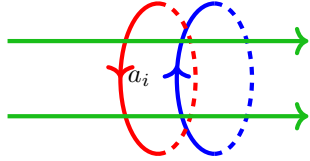
    \begin{center}
        \begin{tikzpicture}
            \draw[red, ultra thick, decoration={markings, mark=at position 0.5 with {\arrow{>}}}, postaction={decorate}] (0,0) arc(90:270:0.5 and 1);
            \draw[blue,ultra thick, decoration={markings, mark=at position 0.5 with {\arrow{<}}}, postaction={decorate}] (0.75,0) arc(90:270:0.5 and 1);
            \draw[dashed, red, ultra thick] (0,0) arc(90:-90:0.5 and 1);
            \draw[dashed, blue, ultra thick] (0.75,0) arc(90:-90:0.5 and 1);
            \draw[->, darkgreen, ultra thick] (-2,-0.5) -- (2,-0.5);
            \draw[->, darkgreen, ultra thick] (-2,-1.5) -- (2,-1.5);
            \node at (-0.25,-1) {$a_i$};
        \end{tikzpicture}
    \end{center}

    The following tensor is assigned to this picture:

    \begin{center}
        \begin{tikzpicture}
            \node at (-1,0) (D) {$\Delta_{a_i}$};
            \node at (-2,0) (e) {$e_{a_i}$};
            \node at (3,0) (M) {$M_{1,1}$};
            \node at (0.5,1) (MU) {$M_{a_i,1}$};
            \node at (0.5,-1) (ML) {$M_{1,a_i}$};
            \node at (4,0) (mu) {$\mu$};
            \node at (2,0.5) (R1) {$R$};
            \node at (2,-0.5) (R2) {$R$};
            \node at (1.5,0.5) {\tiny{1}};
            \node at (1.5,-0.5) {\tiny{1}};
            \draw[<-] (D) -- (e);
            \draw[->] (M) -- (mu);
            \draw[->] (R1) -- (M);
            \draw[->] (R2) -- (M);
            \draw[->] (R2) -- (ML);
            \draw[->] (R1) -- (MU);
            \draw[->] (D) -- (MU);
            \draw[->] (D) -- (ML);
            \draw[->] (MU) -- ++(0,0.5);
            \draw[->] (ML) -- ++(0,0.5);
        \end{tikzpicture}
        
        \begin{tikzpicture}
            \node at (-3,0) {$=$};
            \node at (-1,0) (D) {$\Delta_{a_i}$};
            \node at (-2,0) (e) {$e_{a_i}$};
            \node at (0.5,1) (MU) {$M_{a_i,1}$};
            \node at (0.5,-1) (ML) {$M_{1,a_i}$};
            \node at (4,0) (mu) {$\mu$};
            \node at (3,0) (R) {$R$};
            \node at (2,0) (D2) {$\Delta_1^{op}$};
            \node at (2.6,0.2) {\tiny{1}};
            \draw[->] (e) -- (D);
            \draw[->] (D) -- (MU);
            \draw[->] (D) -- (ML);
            \draw[->] (D2) -- (MU);
            \draw[->] (D2) -- (ML);
            \draw[->] (R) -- (D2);
            \draw[->] (R) -- (mu);
            \draw[->] (MU) -- ++(0,0.5);
            \draw[->] (ML) -- ++(0,0.5);
        \begin{scope}[shift = {(4,0)}]
            \node at (1,0) {$=$};
            \node at (1.5,1) (e) {$e_{a_i}$};
            \node at (2.5,0) (M) {$M_{a_i,1}$};
            \node at (4,0) (D) {$\Delta_{a_i}$};
            \node at (1.5,-1) (R) {$R$};
            \node at (2.5,-1) (mu) {$\mu$};
            \draw[->] (e) -- (M);
            \draw[->] (R) -- (mu);
            \draw[->] (R) -- (M);
            \draw[->] (M) -- (D);
            \draw[->] (D) -- ++(120:1);
            \draw[->] (D) -- ++(60:1);
            \node at (1.7,-0.5) {\tiny{1}};
        \end{scope}
        \begin{scope}[shift = {(-1,-3)}]
            \node at (1,0) {$=$};
            \node at (3,0) (e) {$e_{a_i}$};
            \node at (4,0) (D) {$\Delta_{a_i}$};
            \node at (2,0) (eps) {$\epsilon_1$};
            \node at (1.5,-1) (R) {$R$};
            \node at (2.5,-1) (mu) {$\mu$};
            \draw[->] (e) -- (D);
            \draw[->] (R) -- (mu);
            \draw[->] (R) -- (eps);
            \draw[->] (D) -- ++(120:1);
            \draw[->] (D) -- ++(60:1);
            \node at (1.6,-0.5) {\tiny{1}};
        \end{scope}
        \begin{scope}[shift = {(4,-3)}]
            \node at (1,0) {$=$};
            \node at (2,0) (e) {$e_{a_i}$};
            \node at (3,0) (D) {$\Delta_{a_i}$};
            \draw[->] (e) -- (D);
            \draw[->] (D) -- ++(120:1);
            \draw[->] (D) -- ++(60:1);
            \node at (4.5,0) {$\dim(H_1)$};
        \end{scope}
        \end{tikzpicture}
    \end{center}

    Hence one component of the comultiplication of the cointegral $e_{a_i}$ is assigned to each strand passing through a dotted component. We get an additional contribution of $\dim(H_1)$.

    Now, step (e) does not change the tensor network above since it is invariant under handle slides. We now check the contribution from inserting the parallel $\kappa$ curves to unpaired $\alpha$ curves. The relevant part of the trisection diagram is 

    
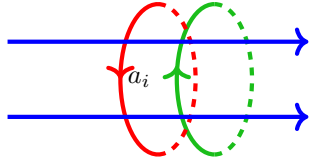
\begin{center}
        \begin{tikzpicture}
            \draw[red, ultra thick, decoration={markings, mark=at position 0.5 with {\arrow{>}}}, postaction={decorate}] (0,0) arc(90:270:0.5 and 1);
            \draw[darkgreen, ultra thick, decoration={markings, mark=at position 0.5 with {\arrow{<}}}, postaction={decorate}] (0.75,0) arc(90:270:0.5 and 1);
            \draw[dashed, red, ultra thick] (0,0) arc(90:-90:0.5 and 1);
            \draw[dashed, darkgreen, ultra thick] (0.75,0) arc(90:-90:0.5 and 1);
            \draw[->, blue, ultra thick] (-2,-0.5) -- (2,-0.5);
            \draw[->, blue, ultra thick] (-2,-1.5) -- (2,-1.5);
            \node at (-0.25,-1) {$a_i$};
        \end{tikzpicture}
    \end{center}

    This clearly has the same contribution to the tensor network as in the previous step, and hence the additional contribution given by introducing the parallel $\kappa$ curve is a factor of $\dim(H_1)$. We can then handle slide the $\alpha$ curves back into the $\alpha$-$\beta$ standard form at no cost. 

    The invariant is thus given by multiplying the assigned elements each component of the link in order of traversal, and evaluating the result with the integral $\mu$, finally multiplying the result of each component together, and normalizing. Note that the normalization factors will take care of the extraneous $\dim(H_1)$ factors. We wish to first compute $\langle T_{\mathrm{st}}\rangle_{\mcal H,e}^G$. Recall the trisection diagram for the stabilization shown in Figure \ref{fig:stabilize_diagram}. The associated tensor is:

    \begin{center}
        \begin{tikzpicture}
            \node at (0,0) (mu1) {$\mu$};
            \node at (1,0) (M1) {$M_{1,1}$};
            \node at (2,0.5) (R1) {$R$};
            \node at (2,-0.5) (e1) {$e_1$};
            \node at (3,0.5) (mu2) {$\mu$};
            \draw[->] (M1) -- (mu1);
            \draw[->] (R1) -- (M1);
            \draw[->] (mu2) -- (R1);
            \draw[->] (e1) -- (M1);
            \begin{scope}[shift = {(4,0)}]
                \node at (0,0) (mu1) {$\mu$};
            \node at (1,0) (M1) {$M_{1,1}$};
            \node at (2,0.5) (R1) {$R$};
            \node at (2,-0.5) (e1) {$e_1$};
            \node at (3,0.5) (mu2) {$\mu$};
            \draw[->] (M1) -- (mu1);
            \draw[->] (R1) -- (M1);
            \draw[->] (mu2) -- (R1);
            \draw[->] (e1) -- (M1);
            \end{scope}
            \begin{scope}[shift = {(8,0)}]
                \node at (0,0) (e) {$e_1$};
            \node at (1,0) (D) {$\Delta_1$};
            \node at (2,0.5) (mu1) {$\mu$};
            \node at (2,-0.5) (mu2) {$\mu$};
            \draw[->] (e) -- (D);
            \draw[->] (D) -- (mu1);
            \draw[->] (D) -- (mu2);
            \end{scope}
            \node at (-1,-2) {$=$};
            \begin{scope}[shift = {(0,-2)}]
                \node at (0,0) (mu) {$\mu$};
                \node at (1,0) (e) {$e_1$};
                \draw[->] (e) -- (mu);
                \node at (2,0) (i) {$i_1$};
                \node at (3,0) (mu2) {$\mu$};
                \draw[->] (i) -- (mu2);
            \end{scope}
            \begin{scope}[shift = {(4,-2)}]
                \node at (0,0) (mu) {$\mu$};
                \node at (1,0) (e) {$e_1$};
                \draw[->] (e) -- (mu);
                \node at (2,0) (i) {$i_1$};
                \node at (3,0) (mu2) {$\mu$};
                \draw[->] (i) -- (mu2);
            \end{scope}
            \begin{scope}[shift = {(8,-2)}]
                \node at (1,0) (mu) {$\mu$};
                \node at (0,0) (e) {$e_1$};
                \draw[->] (e) -- (mu);
                \node at (2,0) (i) {$i_1$};
                \node at (3,0) (mu2) {$\mu$};
                \draw[->] (i) -- (mu2);
            \end{scope}
        \end{tikzpicture}
    \end{center}
    
    Hence $\langle T_{\mathrm{st}}\rangle_{\mcal H,e}^G = \mu(i_1)^3 = \dim(H_1)^3$. Let $|L_1|$ be the number of dotted components in the Kirby diagram, and let $|L_2|$ be the number of components of the link $L$. Since for each parallel $\alpha$-$\beta$ pair, we get a contribution of $\dim(H_1)$, and for each unpaired $\alpha$ curve at step (e), we get a contribution of $\dim(H_1)$, then we can see that \[\dim(H_1)^{|L_1|}\dim(H_1)^{g-|L_2|}\dim(H_1)^{|L_2| - |L_1|}I_{\mcal H}(X,f,L) = \dim(H_1)^gI_{\mcal H}(X,f,L)\] Since $Z((\xi,\tilde x);G;\mcal H,e) = \dim(H_1)^{-g}\langle T(L)\rangle_{\mcal H,e}^G$, we see the invariants indeed match.
\end{proof}

\section{Invariants of 3-Manifolds}
\label{3ManifoldInvariants_Sec}
From an involutory Hopf pair of finite type $(H^\alpha, H^\beta, \langle\_\rangle)$, one can define a 3-manifold invariant in the following way: Let $e^\alpha$ and $e^\beta$ be cointegrals for $H^\alpha_1$ and $H^\beta_1$, respectively. Let $D = (\Sigma, \alpha,\beta)$ be a Heegaard diagram for a closed connected oriented 3-manifold $M$. To each upper circle $\alpha_i$ we assign the tensor 
\begin{center}
    \begin{tikzpicture}
        \node at (0,0) (D) {$\Delta^\alpha_{1,\dots ,1}$};
        \foreach \angle in {-135, -90,45,90,135} {
            \draw[->] (\angle:0.4) -- (\angle:1);
        }
        \node at (135:1.4) {$c_1$};
        \node at (90:1.4) {$c_2$};
        \node at (45:1.4) {$c_3$};
        \node at (-90:1.4) {$c_{n_i-1}$};
        \node at (-135:1.4) {$c_{n_i}$};
        \foreach \angle in {10, -20, -50} {
            \node at (\angle:0.7) {$\cdot$};
        }
        \node at (-2.5,0) (e) {$e^\alpha$};
        \draw[->] (e) -- (D);
    \end{tikzpicture}
\end{center}
where $c_k$ is the $k$th crossing encountered by traversing $\alpha_i$ from the base point along its orientation.

To each $\beta_j$, we assign the tensor
\begin{center}
    \begin{tikzpicture}
        \node at (0,0) (D) {$\Delta^\beta_{1,\dots ,1}$};
        \foreach \angle in {-135, -90,45,90,135} {
            \draw[->] (\angle:0.4) -- (\angle:1);
        }
        \node at (135:1.4) {$c_1$};
        \node at (90:1.4) {$c_2$};
        \node at (45:1.4) {$c_3$};
        \node at (-90:1.4) {$c_{n_j-1}$};
        \node at (-135:1.4) {$c_{n_j}$};
        \foreach \angle in {10, -20, -50} {
            \node at (\angle:1.2) {$\cdot$};
        }
        \node at (-2.5,0) (e) {$e^\beta$};
        \draw[->] (e) -- (D);
    \end{tikzpicture}
\end{center}
where $c_k$ is the $k$th crossing encountered by traversing $\beta_j$ from the base point along its orientation. To each crossing point $p$ of $\alpha_i$ with $\beta_j$, we contract via the tensor
\begin{center}
    \begin{tikzpicture}
        \node at (0,0) (S) {$(S^\beta_{1})^{\nu}$};
        \node at (2,0) (B) {$\bullet$};
        \draw[->] (-2,0) -- (S);
        \draw[->] (S) -- (B);
        \draw[<-] (B) -- (3,0);
    \end{tikzpicture}
\end{center}
where $\nu = 0$ if $(d_p(\alpha_i), d_p(\beta_j))$ forms a positive basis for $T_p\Sigma$, and $\nu = 1$ otherwise. Note that this construction is equivalent to the construction of the invariant of flat bundles over 4-manifolds defined above, ignoring the colored curve set.

Define $\langle D\rangle$ to be the contraction of this tensor network and define \[Z(M) = (\langle e^\alpha, e^\beta\rangle)^{-g(\Sigma)} \langle D\rangle\]

This construction is equivalent to the construct of a 3-manifold invariant given in \cite[Definition 3.12]{CCC22}.

\subsection{Invariants of flat bundles over 3-manifolds}

From an involutory Hopf $G$-doublet of finite type $(H^\alpha, H^\beta, \langle \_\rangle)$, one can define an invariant of flat bundles over 3-manifolds in the following way: let $G$ be a group, let $e^\alpha = \{e^\alpha_g\}_{g\in G}$ be a $G$-cointegral for $H^\alpha$, and let $e^\beta$ be a cointegral for $H^\beta_1$. Let $(\xi,\tilde x)$ be a pointed flat bundle over a 3-manifold $M$ with monodromy $f:\pi_1(M,x)\to G$. Let $D = (\Sigma, \alpha,\beta)$ be a Heegaard diagram colored by $f$ analogously to \cite{V05}. To each upper circle $\alpha_i$, suppose that it is colored by $a_i$ and assign the tensor 

\begin{center}
    \begin{tikzpicture}
        \node at (0,0) (D) {$\Delta^\alpha_{a_i}$};
        \foreach \angle in {-135, -90,45,90,135} {
            \draw[->] (\angle:0.4) -- (\angle:1);
        }
        \node at (135:1.4) {$c_1$};
        \node at (90:1.4) {$c_2$};
        \node at (45:1.4) {$c_3$};
        \node at (-90:1.4) {$c_{n_i-1}$};
        \node at (-135:1.4) {$c_{n_i}$};
        \foreach \angle in {10, -20, -50} {
            \node at (\angle:0.7) {$\cdot$};
        }
        \node at (-2.5,0) (e) {$e^\alpha_{a_i}$};
        \draw[->] (e) -- (D);
    \end{tikzpicture}
\end{center}
where $c_k$ is the $k$th crossing encountered by traversing $\alpha_i$ from the base point along its orientation.

To each $\beta_j$, we assign the tensor
\begin{center}
    \begin{tikzpicture}
        \node at (0,0) (D) {$\Delta^\beta_{b_1,\dots ,b_{n_j}}$};
        \foreach \angle in {-135, -90,45,90,135} {
            \draw[->] (\angle:0.4) -- (\angle:1);
        }
        \node at (135:1.4) {$c_1$};
        \node at (90:1.4) {$c_2$};
        \node at (45:1.4) {$c_3$};
        \node at (-90:1.4) {$c_{n_j-1}$};
        \node at (-135:1.4) {$c_{n_j}$};
        \foreach \angle in {10, -20, -50} {
            \node at (\angle:1.2) {$\cdot$};
        }
        \node at (-2.5,0) (e) {$e^\beta$};
        \draw[->] (e) -- (D);
    \end{tikzpicture}
\end{center}
where $c_k$ is the $k$th crossing encountered by traversing $\beta_j$ from the base point along its orientation, and $b_k$ is defined as follows: first fix a convention for positive and negative crossings as in the previous section and suppose $c_k$ is a crossing with $\beta_j$ and $\alpha_i$. Let $b_k = a_i^{\nu}$ where $\nu = 1$ if $c_k$ is a positive crossing, and $\nu = -1$ if $c_k$ is a negative crossing.

To each crossing of $\alpha_i$ with $\beta_j$, we contract via the tensor
\begin{center}
    \begin{tikzpicture}
        \node at (0,0) (S) {$(S^\beta_{a_i^{(-1)^\nu}})^{\nu}$};
        \node at (2,0) (B) {$\bullet$};
        \draw[->] (-2,0) -- (S);
        \draw[->] (S) -- (B);
        \draw[<-] (B) -- (3,0);
    \end{tikzpicture}
\end{center}
where $\nu = 0$ if the crossing is positive and $\nu = 1$ if the crossing is negative.

Define $\langle D\rangle$ as the contraction of this tensor network and define \[Z(\xi, \tilde x) = (\langle e^\alpha_1, e^\beta\rangle)^{-g(\Sigma)} \langle D\rangle\]

Note that if $H^\beta = H$, then setting $H^\alpha = H^{*,cop}$ with the standard pairing yields the same invariant as in \cite{V05}.

\section{Examples}
\label{Examples_Sec}
\begin{xmp}
    Let $G,H,H',H''$ be finite groups and let $\phi:H\to G$, $\phi':H'\to G$, $\phi'':H''\to G$, $\psi:H\to H'$, and $\psi':H\to H''$ be homomorphisms of groups such that $\phi = \phi'\circ\psi$ and $\phi = \phi''\circ \psi'$. We construct a Hopf $G$-triplet as follows: 
    
    Using the notation and conventions of \cite[Example 4.2]{V05}, let $(e_h)_{h\in H}$ be the standard basis for the Hopf algebras of complex-valued functions on $H$, and similarly denote by $(e'_h)_{h\in H'}$ and $(e''_h)_{h\in H''}$ the Hopf algebras of complex-valued functions on $H'$ and $H''$, respectively. Let $\delta_x^y$ be the Kronecker delta function assigning 1 when $x = y$ and 0 otherwise. Define the Hopf $G$-coalgebra $H^\phi = \{H^\phi_\alpha\}_{\alpha\in G}$ by:
    \[ H^\phi_\alpha = \sum_{g\in \phi^{-1}(\alpha)} \mbb C e_g,\qquad M^\phi_\alpha(e_g\otimes e_h) = \delta_g^h e_g\quad\text{ for any }g,h\in \phi^{-1}(\alpha) \]
    \[ i^\phi_\alpha = \sum_{g\in \phi^{-1}(\alpha)} e_g, \qquad \epsilon^\phi(e_g) = \delta_g^1 \quad\text{ for any }g\in \phi^{-1}(1) \]
    \[ \Delta^\phi_{\alpha,\beta}(e_g) = \sum_{\substack{h\in \phi^{-1}(\alpha) \\ k\in\phi^{-1}(\beta) \\ hk = g}} e_h\otimes e_k \quad\text{ for any }g\in \phi^{-1}(\alpha\beta)\]
    \[ S^\phi_\alpha(e_g) = e_{g^{-1}}\quad\text{ for any }g\in\phi^{-1}(\alpha) \]

    Similarly define $H^{\phi'}$ and $H^{\phi''}$. Note that these Hopf $G$-coalgebras are involutory and of finite type. Now define the following pairings: 

    Assume that $\ker(\phi'') = \langle a\rangle$ is cyclic of order $n$. Let $\mcal F:\mbb C[\ker(\phi'')]\to (\mbb C[\ker(\phi'')])^*$ be the Fourier transform isomorphism defined by the primitive $n$th root of unity $\exp(2\pi i/n)$. Also let $\rho:\ker(\phi'')\to \ker(\phi')$ be a group homomorphism. For $x = a^m\in \ker(\phi'')$ and $y\in \ker(\phi')$,
    \[ \langle e''_x,e'_y\rangle^{\phi'',\phi'} = (\rho\circ\mcal F^{-1}(e''_x))(e'_y) = \frac{1}{n}\sum_{a^p\in\rho^{-1}(y)}e^{\frac{2\pi i}{n}mp}\] we will use the notation $F_{xy} = \langle e_x'', e_y'\rangle^{\phi'',\phi'}$.

    For any $x\in (H^\phi_\alpha)^* = \sum_{g\in \phi^{-1}(\alpha)} \mbb C g$ and any $e'_y\in H^{\phi'}_\alpha$, 
    \[ \langle x,e'_y\rangle^{\phi,\phi'}_\alpha = e'_y(\psi(x)) = \delta_y^{\psi(x)} \]

    Similarly define $\langle \_ \rangle^{\phi,\phi''}_\alpha$ using $\psi'$.

    \begin{lem}\label{ex_triplet}
        If $\ker(\phi')$ and $\ker(\phi'')$ are cyclic of order 2 (or trivial) and are central in $H'$ and $H''$, respectively, then the structures above form a Hopf $G$-triplet.
    \end{lem}

    \begin{proof}
        It is clear that $\langle\_\rangle^{\phi'',\phi'}$ induces a Hopf pair, since all involved maps are group homomorphisms and Hopf algebra isomorphisms, which induce maps of Hopf algebras. We check that $\langle \_\rangle^{\phi,\phi'}$ induces a Hopf $G$-doublet structure, then by symmetry, $\langle\_\rangle^{\phi,\phi''}$ will too.

        We first check $\langle\_\rangle^{\phi'',\phi'}_\alpha$ induces a coalgebra morphism. We will simplify the notation to just $\langle\_\rangle$ for brevity. Let $g,h\in (\phi')^{-1}(\alpha)$ and let $x\in \phi^{-1}(\alpha)$, then
        \[ \langle (M^\phi_\alpha)^*(x), e_g'\otimes e_h'\rangle = \langle x\otimes x, e_g'\otimes e_h'\rangle = \delta_g^{\psi(x)}\delta_h^{\psi(x)} = \delta_g^h \delta_g^{\psi(x)} = \langle x, M^{\phi'}_\alpha(e_g'\otimes e_h')\rangle\]
        
        \[\langle x, \sum_{g\in (\phi')^{-1}(\alpha)} e_g'\rangle = \sum_{g\in (\phi')^{-1}(\alpha)}\delta_g^{\psi(x)} = 1 = (i_\alpha)^*(x)\]

        hence this pairing defines a calgebra morphism. We next check that it intertwines the multiplication, unit, and antipode. Let $z\in (\phi')^{-1}(\alpha\beta)$ and let $x\in \phi^{-1}(\alpha)$, $y\in \phi^{-1}(\beta)$, then 
        \[ \langle x\otimes y, \sum_{\substack{g\in (\phi')^{-1}(\alpha) \\ h\in (\phi')^{-1}(\beta) \\ gh = z}} e_g'\otimes e_h'\rangle = \delta_z^{\psi(xy)} = \langle xy, e_z'\rangle \]

        For $z\in (\phi')^{-1}(1)$,
        \[ \langle 1, e'_z\rangle = \delta_z^{\psi(1)} = \delta_z^1 = \epsilon_1'(e_z')\]

        For $z\in (\phi')^{-1}(\alpha)$, $x\in \phi^{-1}(\alpha^{-1})$,
        \[ \langle x^{-1}, e_y'\rangle = \delta_y^{\psi(x^{-1})} = \delta_{y^{-1}}^{\psi(x)} = \langle x, e_{y^{-1}}'\rangle \]

        To finally prove that we have a Hopf $G$-triplet, we need to prove the equality in part (b) of lemma \ref{fundlemoftriplets_lem}. The LHS of the relation evaluated at $(e''_x\otimes e'_y)$, then post-evaluated at $z\in \phi^{-1}(\alpha)$ is given by 

        \[ \sum_{\substack{h\in \ker(\phi'') \\ k\in \ker(\phi') \\ \psi'(z)h = x \\ k\psi(z^{-1}) = y}} F_{hk}\]

        The RHS evaluated at the same is given by

        \[ \sum_{\substack{h\in \ker(\phi'') \\ k\in \ker(\phi') \\ \psi'(z^{-1})h = x^{-1} \\ k\psi(z) = y^{-1}}} F_{hk}\]

        These two are equal if $\psi'(z)h = h^{-1}\psi'(z)$ and $\psi(z)k^{-1} = k\psi(z)$ for each $x\in \phi''^{-1}(\alpha)$, $y\in \phi'^{-1}(\alpha^{-1})$, $z\in \phi^{-1}(\alpha)$, and each $\alpha\in G$. This obviously holds if the hypotheses hold.
    \end{proof}

    Note that we can extend the pairing between $H^{\phi'}$ and $H^{\phi''}$ when $\ker(\phi')$ and $\ker(\phi'')$ are finite abelian. Then the hypotheses in lemma \ref{ex_triplet} can be extended so that we just require all elements of $\ker(\phi')$ and $\ker(\phi'')$ to have order dividing 2 and are central.

    A $G$-integral for $H^\phi$ is given by $\mu^\phi_{\alpha}(x) = \sum_{a\in \phi^{-1}(\alpha)} a$, and a cointegral for $H^{\phi'}_1$ is $e^{\phi'} = e'_1$. Similarly $e^{\phi''} = e''_1$ is a cointegral for $H^{\phi''}_1$.

    Consider the following trisection diagram for $S^1 \times S^3$

    
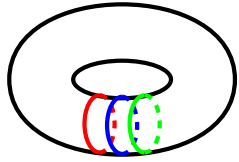
\begin{center}
        \begin{tikzpicture}
        \draw[closed, ultra thick] (0,0) to[curve through = {(1.5,-1) ..  (3,0) .. (1.5,1)}] (0,0);
        \draw[ultra thick] (1.5,0) ellipse (0.65 and 0.25); 
        \draw[red, ultra thick] (1.2,-0.21) arc(90:270:0.2 and 0.38);
        \draw[red, ultra thick, dashed] (1.2,-0.21) arc(90:-90:0.2 and 0.38);
        \draw[blue, ultra thick] (1.5,-0.23) arc(90:270:0.2 and 0.38);
        \draw[blue, ultra thick, dashed] (1.5,-0.23) arc(90:-90:0.2 and 0.38);
        \draw[green, ultra thick] (1.8,-0.21) arc(90:270:0.2 and 0.38);
        \draw[green, ultra thick, dashed] (1.8,-0.21) arc(90:-90:0.2 and 0.38);
    \end{tikzpicture}
    \end{center}

    Let $f:\pi_1(S^1\times S^3,*)\to G$ be the monodromy for a pointed flat bundle over $S^1\times S^3$. Since $\pi_1(S^1\times S^3,*) = \mbb Z$, this amounts to picking an element $\alpha$ of $G$. The trisection bracket reduces to $\mu_\alpha^\phi(i_\alpha) = |\phi^{-1}(\alpha)|$. For each $\alpha$ in the image of $\phi$, this is in bijection with the size of the kernel of $\phi$, but for each $\alpha$ not in the image of $\phi$, this is 0.

    To show that the map $\phi$ need not always be surjective, let $D_n$ denote the dihedral group of order 2n with generators $s$ and $r$ so that $s^2 = r^n = 1$, and let $G = D_8$, $H = H' = H'' = D_4$. Define $\phi' = \phi:D_4\to D_8$ by $(s,r)\mapsto (s,r^4)$. Then the kernel of $\phi'$ is $\langle r^2\rangle$, which is a central copy of $\mbb Z/2\mbb Z$ in $H'$. Define $\phi'':D_4\to D_8$ to be $(s,r)\mapsto (s,r^2)$. Then $\ker(\phi'')$ is trivial. Define $\psi:H\to H'$ by the identity and define $\psi':H\to H''$ by sending $(s,r)\mapsto (s,r^2)$. By lemma \ref{ex_triplet}, these structures give rise to a Hopf $D_8$-triplet where $\phi$ is not surjective. Hence the invariant can distinguish flat bundles.
\end{xmp}

\vspace{0.5cm}
\noindent\textbf{Acknowledgements.} NB and SXC are partially supported 
by National Science Foundation
under Award No. 2006667 and No. 2304990.

\end{document}